\documentclass[12pt]{amsart}
\usepackage[margin=1.1in]{geometry}

\usepackage{appendix}

\usepackage{fullpage}

\usepackage{tikz-cd}

\usepackage{amsmath,amssymb,amsthm} 

\usepackage[all]{xy}
\usepackage{tensor}
\usepackage{amscd} 
\usepackage{mathrsfs}
\usepackage[nobysame,alphabetic,initials,msc-links,lite,backrefs]{amsrefs}

\usepackage{mathscinet}

\makeatletter
\renewcommand{\BibLabel}{%
    \Hy@raisedlink{\hyper@anchorstart{cite.\CurrentBib}\hyper@anchorend}%
    [\thebib]%
}
\makeatother

\DefineSimpleKey{bib}{how}
\renewcommand{\PrintDOI}[1]{%
  \href{http://dx.doi.org/#1}{\texttt{DOI:#1}}%
}
\renewcommand{\eprint}[1]{#1}
\BibSpec{book}{%
    +{}  {\PrintPrimary}                {transition}
    +{.} { \PrintDate}                  {date}
    +{.} { \textit}                     {title}
    +{.} { }                            {part}
    +{:} { \textit}                     {subtitle}
    +{,} { \PrintEdition}               {edition}
    +{}  { \PrintEditorsB}              {editor}
    +{,} { \PrintTranslatorsC}          {translator}
    +{,} { \PrintContributions}         {contribution}
    +{,} { }                            {series}
    +{,} { \voltext}                    {volume}
    +{,} { }                            {publisher}
    +{,} { }                            {organization}
    +{,} { }                            {address}
    +{,} { }                            {status}
    +{,} { \PrintDOI}                   {doi}
    +{,} { \PrintISBNs}                 {isbn}
    +{}  { \parenthesize}               {language}
    +{}  { \PrintTranslation}           {translation}
    +{;} { \PrintReprint}               {reprint}
    +{.} { }                            {note}
    +{.} {}                             {transition}
    +{}  {\SentenceSpace \PrintReviews} {review}
}
\BibSpec{article}{%
    +{}  {\PrintAuthors}                {author}
    +{,} { \textit}                     {title}
    +{.} { }                            {part}
    +{:} { \textit}                     {subtitle}
    +{,} { \PrintContributions}         {contribution}
    +{.} { \PrintPartials}              {partial}
    +{,} { }                            {journal}
    +{}  { \textbf}                     {volume}
    +{}  { \PrintDatePV}                {date}
    +{,} { \issuetext}                  {number}
    +{,} { \eprintpages}                {pages}
    +{,} { }                            {status}
    +{,} { \PrintDOI}                   {doi}
    +{,} { \eprint}        {eprint}
    +{}  { \parenthesize}               {language}
    +{}  { \PrintTranslation}           {translation}
    +{;} { \PrintReprint}               {reprint}
    +{.} { }                            {note}
    +{.} {}                             {transition}
    +{}  {\SentenceSpace \PrintReviews} {review}
}
\BibSpec{collection.article}{%
    +{}  {\PrintAuthors}                {author}
    +{,} { \textit}                     {title}
    +{.} { }                            {part}
    +{:} { \textit}                     {subtitle}
    +{,} { \PrintContributions}         {contribution}
    +{,} { \PrintConference}            {conference}
    +{}  {\PrintBook}                   {book}
    +{,} { }                            {booktitle}
    +{,} { \PrintDateB}                 {date}
    +{,} { pp.~}                        {pages}
    +{,} { }                            {publisher}
    +{,} { }                            {organization}
    +{,} { }                            {address}
    +{,} { }                            {status}
    +{,} { \PrintDOI}                   {doi}
    +{,} { \eprint}        {eprint}
    +{}  { \parenthesize}               {language}
    +{}  { \PrintTranslation}           {translation}
    +{;} { \PrintReprint}               {reprint}
    +{.} { }                            {note}
    +{.} {}                             {transition}
    +{}  {\SentenceSpace \PrintReviews} {review}
}
\BibSpec{misc}{%
  +{}{\PrintAuthors}  {author}
  +{,}{ \textit}      {title}
  +{.}{ }             {how}
  +{}{ \parenthesize} {date}
  +{,} { available at \eprint}        {eprint}
  +{,}{ available at \url}{url}
  +{,}{ }             {note}
  +{.}{}              {transition}
}

\numberwithin{equation}{section}

\newtheorem{theorem}{Theorem}[section]
\newtheorem{lemma}[theorem]{Lemma}
\newtheorem{proposition}[theorem]{Proposition}
\newtheorem{corollary}[theorem]{Corollary}

\theoremstyle{definition}
\newtheorem{definition}[theorem]{Definition}

\theoremstyle{remark}
\newtheorem{rem}[theorem]{Remark}
\newtheorem{remark}[theorem]{Remark}

\newtheorem{example}[theorem]{Example}

\newcommand{\Z}{{\mathbb{Z}}}
\newcommand{\R}{{\mathbb{R}}}
\newcommand{\C}{{\mathbb{C}}}
\newcommand{\T}{{\mathbb{T}}}

\newcommand{\K}{\mathrm{K}} 

\newcommand{\norm}[1]{\left\| #1 \right\|}
\newcommand{\absv}[1]{\left| #1 \right|}

\newcommand{\Zyl}{\mathscr{Z}} 

\DeclareMathOperator{\HP}{HP}

\DeclareMathOperator{\KK}{\mathrm{KK}}
\DeclareMathOperator{\HH}{\mathrm{H}}

\DeclareMathOperator{\ev}{ev}

\DeclareMathOperator{\id}{id}

\title{\normalfont{Smooth Connes--Thom isomorphism, cyclic homology, and equivariant quantisation}} 

\author{Sayan Chakraborty}
\address{Department of Mathematics, Indian Institute of Science Education and Research Bhopal, 
\newline Bhopal Bypass Road, Bhopal, 462066, India, csayan@iiserb.ac.in}
\author{Xiang Tang}
\address{Department of Mathematics and Statistics, Washington University, St. Louis, MO, 63130, U.S.A., xtang@math.wustl.edu}
\author{Yi-Jun Yao} 
\address{School of Mathematical Sciences, Fudan University, Shanghai, China, 200433, \newline yaoyijun@fudan.edu.cn}

\begin{document}





\maketitle 


\begin{center}
Dedicated to Professor Alain Connes on the Occasion of his 70th Birthday
\end{center}

\begin{abstract}
Using a smooth version of the Connes--Thom isomorphism in Grensing's bivariant $K$-theory for locally convex algebras, we prove an equivariant version of the Connes--Thom isomorphism in periodic cyclic homology. As an application, we prove that periodic cyclic homology is invariant with respect to equivariant strict deformation quantization. 
\end{abstract}




\section{Introduction} \label{subsec:main-setup}
In this paper, we are interested in the following $C^*$-dynamical systems. Let $\R^n$ act strongly continuously on a C*-algebra $A$ by an action $\alpha$.  Also let $G$ be a finite group acting on $A$ by $\beta$. Let $\rho: G\to GL_n(\R)$ be a representation of $G$. As $G$ is finite, there is always a $G$-invariant inner product  on $\R^n$. Without loss of generality, we assume in this paper that the $G$ representation $\rho$ preserves the standard metric on $\R^n$.  We assume that the actions $\alpha$ and $\beta$ are compatible, i.e. 
\begin{equation}\label{eq:action}
\beta_g \alpha_x=\alpha_{\rho_g(x)}\beta_g,\ \ \text{for all}\ g\in G,\ x\in \R^n. 
\end{equation}
When $\rho$ is the trivial group homomorphism, the actions $\alpha$ 
and $\beta$  commute. The above condition really means that there is a strongly continuous action of the semi-direct product $\R^n\rtimes G$ on $A$.  With the above data, we consider the crossed product algebra $A\rtimes_\alpha \R^n$, $A\rtimes_\beta G$, and $(A\rtimes _\alpha \R^n) \rtimes _{\beta\rtimes \rho} G$. 

Connes established in \cite{MR605351} a far reaching generalization of the Thom isomorphism theorem in $K$-theory to noncommutative geometry,
\[
\K_\bullet(A)\cong \K_{\bullet+n}(A\rtimes_\alpha \R^n),
\]
with the $G$-equivariant version proved by Kasparov \cite{MR1388299}
\[
\K_\bullet(A\rtimes G) \cong \K_{\bullet} \Big( \big((A\otimes \mathbb{C}_n)\rtimes \R^n\big)\rtimes G\Big),
\]
where $\mathbb{C}_n$ is the complexified Clifford algebra associated to $\mathbb{R}^n$. This result quickly becomes one of the key tool in the study of noncommutative geometry, e.g. \cite{MR866491, MR1303779, MR1388299}. 

In cyclic homology, Elliott, Natsume, and Nest in \cite{MR945014} proved a cyclic version of the Connes--Thom isomorphism. Let $A^{\infty}$ be the smooth subalgebra of $A$ of smooth vectors for the action $\alpha$. Then $A^{\infty}$ is a locally convex algebra with respect to a system of seminorms $(p_i)$, which is induced from the norm of $A$. In the above setting, $G$ also acts on $A^{\infty}.$ So we consider the crossed product algebras $A^{\infty}\rtimes_\beta G$ and $(A^{\infty}\rtimes_{\alpha} \R^n) \rtimes G.$ Elliott, Natsume, and Nest \cite{MR945014}  proved the following identity in periodic cyclic homology. 
\[
\HP_\bullet(A^\infty)\cong \HP_{\bullet+n}(A^\infty\rtimes \R^n). 
\]

In this article, we prove an equivariant version of the cyclic Connes--Thom isomorphism. 
\begin{theorem}\label{thm:hp-connes-thom}
(Corollary \ref{cor:connes-thom-cyclic}) With the above notations and assumptions, we have
\begin{equation}\label{eq:cyclic}
\HP_\bullet\Big(\big((A^{\infty}\otimes\C_n)\rtimes_\alpha
\R^n\big)\rtimes G\Big)\cong \HP_\bullet(A^{\infty}\rtimes_{\beta}G ).
\end{equation}
\end{theorem}
The proof of Elliott, Natsume, and Nest \cite{MR945014} is by induction, which makes it impossible to generalize the proof therein to establish an equivariant Connes--Thom isomorphism in cyclic homology. A new approach different from \cite{MR945014} is in demand to prove Theorem \ref{thm:hp-connes-thom}.

In this paper our approach to Theorem \ref{thm:hp-connes-thom} is completely different from the method in \cite{MR945014}. We present a smooth version of the equivariant Connes--Thom isomorphism using Grensing's bivariant $K$-theory \cite{Grensing12} for locally convex algebras, from which Theorem \ref{thm:hp-connes-thom} follows as a corollary. 

\begin{theorem}\label{thm:equ-connes-thom}(Theorem \ref{thm:equ-thom}) With the assumptions in this section, we have the following equation,
\[
\HH\big(((A^\infty\otimes\C_n)\rtimes_\alpha
\R^n\big)\rtimes G)\cong \HH(A^\infty\rtimes_{\beta} G),
\]
where $\HH$ is a split-exact, diffotopy invariant, $\mathcal{K}^\infty$ stable functor on the category of locally convex algebras, and $\C_n$ is the complexified Clifford algebra associated with
$\R^n$, carrying the trivial action of $\R^n$ and the action of $G$ which is induced from the $G$ action on $\R^n.$
\end{theorem}

The key ingredient in our proof of the above Theorem \ref{thm:equ-connes-thom} is the pseudodifferential calculus for strongly continuous $\mathbb{R}^n$-actions which was introduced by Connes in \cite{connes80}. Using this pseudodifferential calculus with noncommutative symbols, we introduce a smooth version of the Dirac and dual Dirac operator for $\mathbb{R}^n$-actions, which lead us to  the desired isomorphism in Grensing's theory.  Another important ingredient in our proof is an equivariant Takesaki--Takai duality that there is a $G$-equivariant isomorphism between
$(A^\infty\rtimes_\alpha \R^n)\rtimes_{\widehat{\alpha}}\R_n$ and $A^\infty\otimes \mathcal{K}^\infty$, where $\R_n$ is the Pontryagin dual group with a natural action $A^\infty\rtimes_\alpha \R^n$ and $\mathcal{K}^\infty$ is the algebra of smoothing operators on $\R^n$. 

As an application of Theorem \ref{thm:hp-connes-thom}, we study the cyclic homology of equivariant strict deformation quantization introduced by Rieffel \cite{Rieffel93}.

Let $J$ be a real antisymmetric $n\times n$ matrix.  Define $GL_n(J)$ to be the group of invertible
matrices $g$ such that $g^t J g =J$, where $SL_n(\R, J):=SL_n(\R)\bigcap GL_n(J)$.  Assume that $\rho: G\rightarrow GL_n(\R)$ takes value in $SL_n(\R, J)$. Rieffel \cite{Rieffel93} constructed a strict deformation quantization $A^\infty_J$. It was checked in \cite{tang-yao:K-equi-quan} that the actions $\alpha$ and $\beta$ of $\mathbb{R}^n$ and $G$ extend to $A^\infty_J$ with the same properties. Using Theorem \ref{thm:hp-connes-thom}, we prove the following result about cyclic homology of $A^\infty_J\rtimes G$, which is the cyclic analog of the $K$-theory result obtained by the last two authors in \cite{tang-yao:K-equi-quan} (see also \cite{Ech10}). This result can also be viewed as an equivariant generalization of the homotopy invariance of periodic cyclic homology to locally convex algebras, c.f. \cites{MR823176, Getzler, Goodwillie, MR3669113}.
\begin{theorem}\label{thm:cyclic-quantization} (Corollary \ref{cor:cyclic-quantization})
	For a general deformation quantization $A_{J}^\infty$ defined by Equation (\ref{eq:star}), $\HP_\bullet (A_{J}^{\infty}\rtimes G)$ is independent of the $J$ parameter. 
\end{theorem}

This article is organised as follows. In Section \ref{sec:preliminary}, we recall the material about Grensing's bivariant $K$-theory of locally convex algebras. In Section \ref{sec:equ-takai}, we prove an equivariant Takesaki--Takai duality theorem. Theorem \ref{thm:hp-connes-thom} and \ref{thm:equ-connes-thom} are proved in Section \ref{sec:connes-thom}. Theorem \ref{thm:cyclic-quantization} together with examples are presented in Section \ref{sec:quantization}.  In Appendix \ref{sec:equi_Grensing}, we develop an equivariant version of Grensing's results \cite{Grensing12}, which is needed in our proofs. In Appendix \ref{sec:calculus}, we briefly recall Connes' pseudodifferential calculus for $\R^n$-actions, and some analytic properties about the symbol of the dual Dirac operator introduced in the proof of Theorem \ref{thm:equ-connes-thom} is discussed in Appendix \ref{app:schwarz}. \\

\noindent{\bf Acknowledgments}: This work was started when Chakraborty was a Ph.D. student at M\"unster, Germany. He would like to thank Joachim Cuntz and Siegfried Echterhoff for useful discussions. Tang's research is partially supported by NSF, and Yao's research is partially supported by NSFC. Part of this work is completed during Tang's visit of Fudan University and Shanghai Center of Mathematical Sciences. He would like to thank the two institutes for their warm hospitality hosting his visits.  
 
\section{Preliminaries}\label{sec:preliminary}
In this section, we briefly recall some preliminary material about the bivariant $K$-theory of (graded) locally convex algebras. We remark that analogous to the equivariant $KK$-theory \cite{MR1388299}, it may seem natural to introduce and work with a $G$-equivariant version of Grensing's bivariant $K$-theory. To avoid the technical difficulty in defining an equivariant Kasparov module for locally convex algebras, we have chosen to work with the nonequivariant version of bivariant $K$-theory  \cite{Grensing12} in this paper.   
\subsection{Locally convex algebra}
Recall that a locally convex algebra $\mathcal{A}$ is a complete locally convex space endowed with a continuous multiplication $\mathcal{A}\times \mathcal{A}\to \mathcal{A}$ such that for every continuous seminorm $\mathfrak{p}$ on $\mathcal{A}$ there is a continuous seminorm $\mathfrak{q}$ on $\mathcal{A}$ such that for all $a,b\in \mathcal{A}$,
\[
\mathfrak{p}(ab)\leq \mathfrak{q}(a)\mathfrak{q}(b).
\]
Through out the paper, we will always work with projective tensor product of locally convex algebras. By abusing notation, we will use the symbol $\mathcal{A}\otimes\mathcal{B}$ to denote the projective tensor product between two locally convex algebras $\mathcal{A}$ and $\mathcal{B}$.  When a locally convex algebra $\mathcal{A}$ is equipped with a grading operator $\epsilon\in  \operatorname{Aut}(\mathcal{A})$ such that $\epsilon^2=1$, we call $(\mathcal{A},\epsilon)$ a graded locally convex algebra. 

Let $\mathcal{A}$ be a (graded) locally convex algebra. We will use $\mathcal{S}(\mathbb{R}^n, \mathcal{A})$ to denote the locally convex space of Schwartz functions on $\mathbb{R}^n$ with value in $\mathcal{A}$. Suppose that $\mathbb{R}^n$, as an abelian group, acts smoothly isometrically on the algebra $\mathcal{A}$ by the group homomorphism $\alpha: \mathbb{R}^n\to \operatorname{Aut}(\mathcal{A})$. The following twisted convolution product $\ast$ makes $\mathcal{S}(\mathbb{R}^n, \mathcal{A})$ into a locally convex algebra, i.e. 
\[
(f * f')(x) = \int_{\mathbb{R}^n}f(y)\alpha_{y}(f'(x-y))dy,
\]
where $dy$ is the Lebesgue measure on $\mathbb{R}^n$.  This algebra will be denoted by $\mathcal{A}\rtimes \mathbb{R}^n$. 

An automorphism $u\colon \mathcal{A}\to \mathcal{A}$ is said to be \emph{almost isometric} if for all seminorms $\norm{\cdot}_\alpha,$ there exists a positive constant $C_\alpha$ such that 
\[
\norm{u(a)}_\alpha \le C_\alpha\norm{a}_\alpha, \quad \forall a\in \mathcal{A}.
\] 
We assume that $G$ is a finite group or a compact group that acts on $\mathcal{A}$ from the left by almost isometric automorphisms. And by a $G$-locally convex algebra, we mean a locally convex algebra equipped with an almost isometric $G$-action. 

Now suppose that a finite group $G$ acts on $\mathcal{A}$ from the left by almost isometric automorphisms. 
\begin{definition}[{cf.\ \cite[Sec.\ 7]{Perrot08}}]
	 We define the \emph{smooth crossed product} $\mathcal{A}\rtimes G$ as the completion of the vector space spanned by the elements $(a,g)\in \mathcal{A}\oplus G$ in the system of seminorms
	$$
		\norm{\sum_{g\in G}(a_g,g)}_{\alpha} = C_\alpha\sum_{g\in G} \norm{g^{-1}(a_g)}_\alpha,$$  
	 with multiplication given by
	$$
		(a,g)\cdot (a',g') = (ag(a'),gg'), \quad a,a'\in \mathcal{A}, \; g,g'\in G.
	$$
\end{definition}
It can be checked directly that the seminorms $\norm{\cdot}_{\alpha}$ make $\mathcal{A}\rtimes G$ into a locally convex algebra. Furthermore, if $\mathcal{A}$ is graded, then $\mathcal{A}\rtimes G$ is a graded locally convex algebra. Note that, for notational conveniences, we did not distinguish between seminorms of $\mathcal{A}\rtimes G$ and seminorms of $\mathcal{A}.$ 

\subsection{$KK$-theory for locally convex algebras}
In this subsection we briefly recall the smooth bivariant $K$-theory introduced by Grensing \cite{Grensing12}. 

\subsubsection{Abstract Kasparov module} 
 Let $\mathcal{A}$ and $\mathcal{B}$ be (graded) locally convex algebras. An abstract Kasparov module from $\mathcal{A}$ to $\mathcal{B}$ is given by a quadruple $(\alpha, \bar{\alpha}, U, \hat{\mathcal{B}})$, where $\hat{\mathcal{B}}$  is a (graded) algebra containing $\mathcal{B}$, and $U\in\hat{\mathcal{B}}$ is an invertible element, and $\alpha, \bar{\alpha}:\mathcal{A}\to\hat{\mathcal{B}}$  are two (degree 0) morphisms such that the map
$$
 \mathcal{A}\to \hat{\mathcal{B}};\; a\to \alpha(a)-U^{-1}\bar{\alpha}(a)U
$$
is a $\mathcal{B}$-valued continuous map.

\subsubsection{Locally convex Kasparov module} 
 Let $\mathcal{A}$ and $\mathcal{B}$ be (graded) locally convex algebras. A locally convex Kasparov module from $\mathcal{A}$ to $\mathcal{B}$ is given by a (graded) locally convex algebra $\hat{\mathcal{B}}$, whose grading is defined by an element $\epsilon\in \hat{\mathcal{B}}$ with $\epsilon^2=1$, containing $\mathcal{B}$ as a graded subalgebra, an odd element $F\in\hat{\mathcal{B}}$  and a (degree 0) map $\phi:\mathcal{A}\to\hat{\mathcal{B}}$ such that the maps
 \begin{itemize}
\item $\mathcal{B}\to \hat{\mathcal{B}};\; b\to bF,Fb, \epsilon b, b\epsilon$,
 \item  $\mathcal{A}\otimes \mathcal{B}\to \hat{\mathcal{B}};\;(a,b)\to \phi(a)b,b\phi(a)$,
 \item  $\mathcal{A}\to\hat{\mathcal{B}};\; a\to \phi(a)(1-F^2),(1-F^2)\phi(a)$, 
 \item  $\mathcal{A}\to\hat{\mathcal{B}};\; [\phi(a),F]$
\end{itemize}
are all $\mathcal{B}$-valued and continuous. We denote this locally convex Kasparov $(\mathcal{A},\mathcal{B})$ module by $(\hat{\mathcal{B}},\phi,F)$.

\subsubsection{Quasihomomorphism} 
 Let $\mathcal{A}$ and $\mathcal{B}$ be (graded) locally convex algebras. Let $\hat{\mathcal{B}}$ be a (graded) locally convex algebra containing $\mathcal{B}$. A quasihomomorphism from $\mathcal{A}$ to $\mathcal{B}$ is given by a pair of  (degree 0) maps $(\alpha, \bar{\alpha})$ from $\mathcal{A}$ to $\hat{\mathcal{B}}$ such that
 
  \begin{itemize}
\item $\mathcal{A}\to \hat{\mathcal{B}};\; a\to \alpha(a)-\hat{\alpha}(a)$,
 \item  $\mathcal{A}\otimes \mathcal{B}\to \hat{\mathcal{B}};\;(a,b)\to \alpha(a)b,b\alpha(a)$,

\end{itemize}
 are all $\mathcal{B}$-valued maps and continuous.
 
 Let $\mathcal{A}$, $\mathcal{B}$ be (graded) locally convex algebras. The \emph{cylinder algebra} is the subalgebra $\Zyl \mathcal{A}\subseteq C^\infty([0,1], \mathcal{A})$ consisting of $\mathcal{A}$-valued functions whose derivatives of order $\ge 1$ vanish at the endpoints. Since $\Zyl:=\Zyl\C$ is nuclear, therefore $\Zyl \mathcal{A} = \Zyl\otimes \mathcal{A}$. 

Two morphisms $\phi_0,\phi_1\colon \mathcal{A}\to \mathcal{B}$ are called \emph{diffeotopic} if there exists a (degree 0) morphism $\phi\colon \mathcal{A}\to \Zyl \mathcal{B}$ called \emph{diffotopy}, such that $\ev_0\circ \phi = \phi_0$, $\ev_1\circ\phi = \phi_1$. The relation of diffotopy is clearly an equivalence relation. 

If $\mathcal{A}$ carries an action $\alpha$ of $G$, then we define the action of $G$ on $\Zyl \mathcal{A}= \Zyl\otimes \mathcal{A}$  to be $\id \otimes \alpha.$ 


According to  \cite[Definition 18]{Grensing12}, if $\HH$ is a  split-exact functor on the category of locally convex algebras, a quasihomomorphism $(\alpha, \bar{\alpha})$ from $\mathcal{A}$ to $\mathcal{B}$ gives rise to a map from $\HH(\mathcal{A})$ to $\HH(\mathcal{B})$. 



An abstract Kasparov module from $\mathcal{A}$ to $\mathcal{B}$ gives rise to a quasihomomorphism. Suppose that we have an abstract Kasparov $\mathcal{A}$ to $\mathcal{B}$ module $x:=(\alpha, \bar{\alpha}, U, \hat{\mathcal{B}})$. This immediately gives a quasihomomorphism from $\mathcal{A}$ to $\mathcal{B}$ given by $Qh(x) : = (\alpha, U^{-1}\bar{\alpha}U)$. 
Also if we have a locally convex Kasparov $(\mathcal{A},\mathcal{B})$ module by $(\hat{\mathcal{B}},\phi,F)$, we can set elements $P_\epsilon : = \frac{1}{2}(1+\epsilon), P_\epsilon^\perp := 1- P_\epsilon$ and define an abstract Kasparov module from $\mathcal{A}$ to $\mathcal{B}$ which is given by the quadruple $(P_\epsilon\phi, P_\epsilon^\perp \phi, W_F, \hat{\mathcal{B}})$, where 

$$W_F := P_\epsilon F + P_\epsilon^\perp F(2-F^2) + \epsilon(1-F^2).$$ 

See \cite[Proposition 39]{Grensing12} for the details of the previous discussion. 

We will use the following results of Grensing from \cite{Grensing12}. We remark that Grensing \cite{Grensing12} worked with trivially graded locally convex algebras. However, it is not hard to see that all the developments \cite{Grensing12} naturally extend to graded locally convex algebras with the above definitions with essentially identical proofs. So we will use these results for graded locally convex algebras without detail proofs. In the following, by a morphism between graded algebras we always mean a graded morphism.    

\begin{proposition}(\cite[Proposition 22]{Grensing12})\label{quasirules} Let $(\alpha,\bar\alpha):\mathcal{A}\rightarrow \hat {\mathcal{B}}\trianglerighteq \mathcal{B}$ be a quasihomomorphism of (graded) locally convex algebras, $\HH$ a split exact functor  on the category of (graded) locally convex algebras with value in abelian groups. Then the following properties hold:
\begin{enumerate}
\item\label{quasiprecompose} $\HH((\alpha,\bar\alpha)\circ\phi)=\HH(\alpha,\bar\alpha)\circ \HH(\phi)$ for every morphism $\phi:\mathcal{C}\to \mathcal{A}$,
\item\label{quasipostcompose} $\HH(\psi\circ(\alpha,\bar\alpha))=\HH(\psi)\circ \HH(\alpha,\bar\alpha)$ for every homomorphism $\psi:\hat {\mathcal{B}}\to \hat {\mathcal{C}}$ of  locally convex algebras, such that there is a locally convex subalgebra  $\mathcal{C}\subseteq\hat {\mathcal{C}}$ such that $(\psi\circ\alpha,\psi\circ\bar\alpha):\mathcal{A}\rightarrow \hat {\mathcal{C}}\trianglerighteq {\mathcal{C}}$ is a quasihomomorphism,
\item\label{onemoreprop} if $(\beta,\bar\beta):\mathcal{A}'\rightarrow \hat{\mathcal{B}}'\trianglerighteq \mathcal{B}'$ is a quasihomomorphism, and $\phi:\mathcal{A}\to \mathcal{A}'$, $\psi:\mathcal{B}\to \mathcal{B}'$ are homomorphisms of locally convex algebras such that $\psi\circ (\alpha-\bar\alpha)=(\beta-\bar\beta)\circ \phi$, then
$$ \HH(\psi)\circ \HH(\alpha,\bar\alpha)=\HH(\beta,\bar\beta)\circ \HH(\phi),$$
\item\label{negativequasi} $\HH(\alpha,\bar\alpha)=-\HH(\bar\alpha,\alpha)$,
\item\label{orthosum} if $\alpha-\bar\alpha$ is a homomorphism orthogonal to $\bar\alpha$, then $\HH(\alpha,\bar\alpha)=\HH(\alpha-\bar\alpha)$.
\end{enumerate}
\end{proposition}

If $(\alpha,\bar\alpha):\mathcal{A}\rightarrow \hat {\mathcal{B}}\trianglerighteq \mathcal{B}$ is a quasihomomorphism of (graded) locally convex algebras, the sum $\mathcal{D}:=\mathcal{A}\oplus {\mathcal{B}}$, equipped with the multiplication 
$$(a,b)(a',b'):=(aa',\alpha(a)b'+b\alpha(a')+bb')$$
is a (graded) locally convex algebra $\mathcal{D}_\alpha$\index{$\mathcal{D}_\alpha$}. ${\mathcal{B}}$ sits in $\mathcal{D}_\alpha$ by the inclusion $\iota_{\mathcal{B}}:{\mathcal{B}}\to \mathcal{D}_\alpha,\; b\mapsto (0,b)$, and $\mathcal{A}$ is a quotient of $\mathcal{D}_\alpha$.
Also we have the following split exact sequence
\[\xymatrix{ 0\ar[r]&{\mathcal{B}}\ar[r]&\mathcal{D}_\alpha\ar[r]&\mathcal{A}\ar[r]&0,}\]
with split $\alpha':=\id_{\mathcal{A}}\oplus 0$.  Note that there is another split $\bar{\alpha}':= \id_{\mathcal{A}}\oplus (\bar\alpha-\alpha).$ By abuse of notations, we denote them by $\alpha$ and $\bar{\alpha}$ again, respectively.

\begin{proposition}\cite[Lemma 19]{Grensing12}\label{morphofdoublesplit} If $(\phi_1,\phi_2,\phi_3)$ is a morphism of double split extensions, i.e.:
\begin{equation}\label{homologycommutes}\xymatrix{ 0\ar[r]&{\mathcal{B}}\ar[d]_{\phi_1}\ar[r]^\iota &\mathcal{D}_\alpha\ar[d]_{\phi_2}\ar[r]|{\,\,\pi\,}&{\mathcal{A}}\ar@/_.3cm/[l]_{\alpha}\ar@/^.3cm/[l]^{\bar\alpha}\ar[d]^{\phi_3}\ar[r]&0\\
0\ar[r]&{\mathcal{B}}'\ar[r]^{\iota'}&\mathcal{D}'_\beta \ar[r]|{\,\,\pi'\,}&{\mathcal{A}}'\ar@/_.3cm/[l]_{\beta}\ar@/^.3cm/[l]^{\bar\beta}\ar[r]&0
}\end{equation}
which commutes in the usual sense satisfying $\phi_2\circ\alpha=\beta\circ\phi_3$, $\phi_2\circ\bar\alpha=\bar\beta\circ\phi_3$, then for every split exact functor $\HH$
$$\HH(\phi_1)\circ \HH(\alpha,\bar\alpha)=\HH(\beta,\bar\beta)\circ \HH(\phi_3).$$
\end{proposition}

\subsection{Notations}
In this paper, we shall denote (graded) $C^*$-algebras by $A, B, \cdots,$ and (graded) locally convex algebras by $\mathcal{A}, \mathcal{B}, \cdots$. If $A$ is equipped with a strongly continuous $\mathbb{R}^n$ action, we denote the subalgebra of $A$ of smooth vectors with respect to an $\mathbb{R}^n$-action by $A^\infty$, which is a locally convex algebra.  $G$ will always be a finite group. All the tensor products we use are projective tensor products of (graded) locally convex algebras. For  the convenience to readers, we give a short list of notations:

\begin{itemize}
\item $e(s)$ denotes the number $e^{2 \pi i s}.$
    \item $\mathcal{K}^\infty$ denotes the (trivially graded) locally convex algebra of smooth compact operators (smoothing operators on $\R^n$), and $\mathcal{K},$ the C*-algebra of compact operators. 
 \item $\HH$ is a split-exact, diffotopy invariant, $\mathcal{K}^\infty$ stable functor on the category of (graded) locally convex algebras.  
  \item $AKM(\mathcal{A},\mathcal{B})$ denotes the set of locally convex Kasparov modules between (graded) locally convex algebras $\mathcal{A}$ and $\mathcal{B}$.    
  \item $\KK^G(A,B)$ denotes the usual equivariant Kasparov group for the C*-algebras $A$ and $B$. We denote the non-equivariant one  by $\KK(A,B).$ 
  \item $QH(\mathcal{A},\mathcal{B})$ denotes the quasihomomorphim group between $\mathcal{A}$ and $\mathcal{B}.$  $Qh(x)$ denotes the quasi-homomorphism group associated to the locally convex Kasparov module $x$.
  \item If $x\in QH(\mathcal{A},\mathcal{B})$, then $\HH(x)$ denotes the map from $\HH(\mathcal{A})$ to $\HH(\mathcal{B})$.
 \end{itemize}
 
Let $R(G)$ be the representation ring of the group $G$. We will introduce an $R(G)$-module structure on $\HH(\mathcal{K}^\infty\rtimes G)$ in Appendix \ref{sec:equi_Grensing} for a diffotopy invariant, split exact, $\mathcal{K}^\infty$ stable functor $\HH$. We will need the following $G$-equivariant variation of Grensing's result in the study of the Connes--Thom isomorphism.  

Recall that $G$ acts on $\mathbb{R}^n$ and therefore acts on $\mathcal{K}^\infty$. Let  $(\mathcal{H}, \varphi, \tilde{F})$ be a $G$-equivariant Kasparov $(\C, \mathcal{K})$ module for $C^*$-algebras. We assume that $(\mathbb{B}(\mathcal{H}), \varphi, \tilde{F})$ defines a locally convex Kasparov $(\C, \mathcal{K}^\infty)$ module.   By the descent construction in equivariant $KK$-theory, $(\mathcal{H}, \varphi, \tilde{F})$ gives rise to a Kasparov $(\C G, \mathcal{K}\rtimes G)$ module for $C^*$-algebras in the following way.  Endow $C(G, \mathcal{H})$ , which we denote by $\mathcal{H}\rtimes G$, with the following operations:
\[
\begin{split}
\langle x, y \rangle _{\mathcal{K}\rtimes G}(t):= &\sum_G\beta_{s^{-1}}\big(\langle x(s),y(st)\rangle _{\mathcal{K}}\big),\\
(x \cdot \lambda)(t):= &\sum_G x(s)\beta_s\big(\lambda(s^{-1}t)\big),
\end{split}
\]
for $x,y\in C(G, \mathcal{H}):=\mathcal{H}\rtimes G$, and $\lambda\in \mathcal{K}\rtimes G$. Define $F$ on $C(G, \mathcal{H})$ by $F(x)(t)=\tilde{F}(x(t))$, and $\phi:\C G\to \mathbb{B}(\mathcal{H}\rtimes G)$ by
\[
\big(\phi(f)x\big)(t):=\sum_s\varphi\big(f(s)\big)\gamma_s\big(x(s^{-1}t)\big),
\]
for $f\in \C G$, $x\in \mathcal{H}\rtimes G$. With the assumption that $(\mathbb{B}(\mathcal{H}), \varphi, \tilde{F})$ is a locally convex Kasparov $(\C, \mathcal{K}^\infty)$ module, it is straightforward to check  that the triple $\big(\mathbb{B}(\mathcal{H}\rtimes G), \phi,  F \big)$ defines a locally convex Kasparov $(\C G , \mathcal{K}^\infty\rtimes G )$ module.  We have the following property generalizing \cite[Proposition 45]{Grensing12}.
\begin{proposition}\cite[Proposition 45]{Grensing12}\label{prop:G-index}
 	Let $x=(\mathbb{B}(\mathcal{H}\rtimes G), \phi, F)$ be a locally convex Kasparov $(\C G , \mathcal{K}^\infty\rtimes G )$ module defined above from a $G$-equivariant Kasparov module $(\mathcal{H}, \varphi, \tilde{F})$ for $C^*$-algebras $(\C, \mathcal{K})$, and $\HH$ be a diffotopy invariant, split exact, $\mathcal{K}^\infty$ stable functor. Then  $\HH(Qh(x)) = \operatorname{index}_G \tilde{F} \circ \theta_*$, where $\theta_* : \HH(\C G) \rightarrow \HH(\mathcal{K}^\infty\rtimes G)$ denotes the $G$-stabilisation map, and $\operatorname{index}_G \tilde{F}$ is the $G$-index of the operator $\tilde{F}$ which is an element of $R(G).$
 \end{proposition}

The proof of the above proposition is presented in Appendix \ref{sec:equi_Grensing}, Proposition \ref{prop:index}. 

\section{The equivariant Takesaki--Takai duality theorem} \label{sec:equ-takai}

We consider a $C^*$-dynamical system $(A, \mathbb{R}^n, \alpha)$. Let $\R_n$ be the Pontryagin dual group of $\R^n$ as an abelian group. We observe that there is a dual action $\hat{\alpha}$ of $\R_n$ on  $A^\infty \rtimes_\alpha \R^n$ given by 
 
 $$ \hat{\alpha}_x(f)(s) = e(\langle x ,s\rangle)f(s).$$
 
 The action of $G$ on a smoothing operator (viewed as an operator on $L^2(\R^n)$) is $g(T')(f)(x)=T'(g^{-1}\cdot(f))(g^{-1}x)$. If we realise a compact operator by a kernel function $k(r,s)$ on $\R^n\times \R^n$ then the $G$ action is the diagonal action i.e $g\cdot k(x,y)=k(g^{-1}x,g^{-1}y)$

 Following \cite[Page 190]{Williams07}, we prove a $G$-equivariant version of the Takesaki--Takai duality theorem. 
 
 \begin{theorem}\label{thm:takai}
 	$(A^\infty  \rtimes_\alpha \R^n)\rtimes_{\hat{\alpha}} \R_n$ is isomorphic to $A^\infty  \otimes \mathcal{K}^\infty$. And the isomorphism can be made $G$-equivariant.
  \end{theorem}
  \begin{proof}
Let $\gamma$ be the action of $\R^n$ on $\mathcal{S}(\R^n,A^\infty )$ given by 
 	
 	$$ (\gamma_t f )(s) = f(s-t).$$
 	From the proof of  \cite[Theorem 7.1]{Williams07}, we have an isomorphism of $(A^\infty  \rtimes_\alpha \R^n)\rtimes_{\hat{\alpha}} \R_n$ and $\mathcal{S}(\R^n,A^\infty )\rtimes_\gamma \R^n$ given by $\Phi$, where 
 
 $$\Phi(F)(s,r) = \int_{\R_n}\alpha_r^{-1}(F(t,s))e(\langle r-s,t \rangle)dt, \hspace{.2cm} F \in \mathcal{S}(\R_n \times \R^n, A). $$ 
 In the following, we show that the above isomorphism (\cite[Lemma 7.6]{Williams07}) between $\mathcal{S}(\R^n,A^\infty )\rtimes_\gamma \R^n$ and $A^\infty  \otimes \mathcal{K}^\infty$ is  $G$-equivariant, i.e. $g\cdot \Phi(F) = \Phi(g\cdot F)$.
 
 	  \begin{align*}
 		 \Phi(g\cdot F)(s,r) &= \int_{\R_n}\alpha_r^{-1}(g\cdot F(x,s))e(\langle r-s,x \rangle)dx
 		 \\&= \int_{\R_n}\alpha_r^{-1}(\beta_g(F(g^tx,g^{-1}s)))e(\langle r-s,x \rangle)dx
 		 \\\Phi(g\cdot F)(gs,r)&= \int_{\R_n}\alpha_r^{-1}(\beta_g(F(g^tx,s)))e(\langle r-gs,x \rangle)dx
 		 \\&= \int_{\R_n}\alpha_r^{-1}(\beta_g(F(x,s)))e(\langle r-gs,(g^t)^{-1}x \rangle)dx
 		  \\&= \int_{\R_n}\alpha_r^{-1}(\beta_g(F(x,s)))e(\langle r-gs,(g^t)^{-1}x \rangle)dx
 		  \\&= \int_{\R_n}\alpha_r^{-1}(\beta_g(F(x,s)))e(\langle g^{-1}r-s,x \rangle)dx
		   \end{align*}
	\begin{align*}
 		  \\\Phi(g\cdot F)(gs,gr)&= \int_{\R_n}\alpha_{gr}^{-1}(\beta_g(F(x,s)))e(\langle r-s,x \rangle)dx
 		   \\&= \int_{\R_n}\beta_g(\alpha_{r}^{-1}(F(x,s)))e(\langle r-s,x \rangle)dx
 		     \\\Phi(g\cdot F)(s,r)&= \int_{\R_n}\beta_g(\alpha_{g^{-1}r}^{-1}(F(x,g^{-1}s)))e(\langle g^{-1}r-g^{-1}s,x \rangle)dx
 		     \\&= \beta_g(\Phi(F)(g^{-1}s,g^{-1}r))
 		   \\&= (g\cdot \Phi(F))(s,r).
 \end{align*}

\end{proof}

\begin{corollary}
	The above isomorphism gives a $G$-equivariant isomorphism between the C*-algebras $(A \rtimes_\alpha \R^n)\rtimes_{\hat{\alpha}} \R_n$ and $A  \otimes \mathcal{K}$. And we have the following isomorphism of algebras
	\[
	\big((A \rtimes_\alpha \R^n)\rtimes_{\hat{\alpha}} \R_n\big)\rtimes G\cong \big(A  \otimes \mathcal{K}\big)\rtimes G.
	\]
\end{corollary}

\section{Smooth Connes--Thom isomorphism}\label{sec:connes-thom}

In \cite{MR605351}, Connes proves a Thom isomorphism theorem in $K$-theory, which offers a fundamental tool in noncommutative geometry. More precisely, let $A$ be a (trivially graded) $C^*$-algebra with a strongly continuous $\mathbb{R}^n$-action.  Then 
\[
\K_\bullet(A)\cong \K_{\bullet+n}(A\rtimes \mathbb{R}^n). 
\]

In this paper, we present a smooth version of the above Connes--Thom isomorphism for locally convex algebras. 
\begin{theorem} [Equivariant Connes--Thom isomorphism]\label{thm:equ-thom}Let $\R^n, G,$ and $A$ be defined as in Sec. \ref{subsec:main-setup}. Then
\[
\HH\big(((A^\infty\otimes\C_n)\rtimes_\alpha
\R^n\big)\rtimes G)\cong \HH(A^\infty\rtimes_{\beta} G),
\]
where $\HH$ is a split-exact, diffotopy invariant, $\mathcal{K}^\infty$ stable functor on the category of (graded) locally convex algebras, and $\C_n$ is the complexified Clifford algebra associated with
$\R^n$, carrying the trivial action of $\R^n$ and the action of $G$ which is induced from the $G$ action on $\R^n.$
\end{theorem}

Let $B^\infty := A^\infty\otimes\C_n$. We extend the action $\alpha$ of $\mathbb{R}^n$ on $A$ (and $A^\infty$) to $B^\infty$ by a trivial action on the component $\C_n$ and denote the action again by $\alpha$. We construct a pair of Dirac and dual Dirac elements $x^{{\infty}}_{n,\alpha}$ and $y^{{\infty}}_{n,\alpha}$  as a locally convex  Hilbert module from $(B^\infty\rtimes_\alpha\R^n)\rtimes G$ to $(A^\infty\otimes \mathcal{K}^\infty)\rtimes G$ and  from $A^\infty\rtimes G$ to $(B^\infty\rtimes_\alpha\R^n)\rtimes G$ respectively, which are inverses to each other by the functor $\HH$. When
the algebra $A$ is  $\C$, then the two elements
\[
x_n^{{\infty}}\in
AKM\Big((\mathcal{S}(\R^n)\otimes \C_n)\rtimes G, \mathcal{K}^\infty\rtimes G\Big),\ \text{and}\ y_n^{{\infty}}\in AKM\Big(\C G, (\mathcal{S}(\R^n)\otimes \C_n)\rtimes G \Big),
\]
modify Grensing's construction \cite{Grensing12} to allow the generalization to general $A$.

We recall the standard Hodge-de Rham operator as  $d+d^*$ on $\R^n$, where we
equip $\R^n$ with the standard metric and $d^*$ is the adjoint
of $d$. Our key observation is that the principal symbol of $d+d^*$ is the the Clifford multiplication of $\xi$ on $\wedge^\bullet{\C^n}^*$, which is independent of $x$ and only a function of $\xi$. Using the symbol calculus
by Connes \cite{connes80}, we can generalize the Hodge-de Rham operator to strongly continuous $\R^n$ actions on $C^*$-algebras, and therefore leads to generalizations of 
$x_n$ and $y_n$.  

\subsection{A (dual) Dirac element}\label{subsec:dirac}

Following \cite[Section 10.6]{Higson-Roe:analytic-K}, we use a normalizing function to construct the Dirac element. 

Recall that a normalizing function is a smooth function $\chi:\mathbb{R}\to[-1,1]$ satisfying the following properties
\begin{enumerate}
\item $\chi$ is odd,
\item $\chi(\lambda)>0$ for all $\lambda>0$,
\item $\chi(\lambda)\to \pm 1$ as $\lambda\to \pm \infty$.
\end{enumerate} 

An explicit construction of a normalizing function $\chi$ is constructed in \cite[Ex. 10.9.3]{Higson-Roe:analytic-K} with an extra nice property that the  Fourier transform $\widehat{\chi}$ has compact support and $s\widehat{\chi}(s)$ is a smooth function on $\mathbb{R}$.  

Given $\xi\in \mathbb{R}^n$, we use $c(\xi)$ to denote the Clifford multiplication with respect to $\xi$. We use $e^{is c(\xi)}$ to denote the wave operator on $\mathbb{C}_n$ defined by the wave equation 
\[
\frac{d}{ds} f=isc(\xi)(f),\ 
\]
on $\mathbb{C}_n$, where $i$ is the square root of $-1$. 

Let  $\chi$ be a normalizing function. Define an endomorphism $\chi\big( c(\xi)\big)$ on $\mathbb{C}_n$ by 
\[
\chi\big( c(\xi) \big):=\int_{\mathbb{R}} \widehat{\chi}(s)e^{isc(\xi)} ds. 
\]

Consider an order zero symbol $\Sigma(\xi)$ as follows,
\[
\Sigma(\xi):=1\otimes \chi\big(c(\xi)\big),
\]
where $c(\xi)$ is the Clifford multiplication of $\xi$ on $\mathbb{C}_n$. 

A crucial property here is that
$\Sigma(x,\xi)$ is independent of $x$ and only a function of $\xi$,
which allows us to apply Connes' pseudo-differential calculus, which is reviewed in Appendix \ref{sec:calculus}.

For a $C^*$-dynamical system $(A, \R^n, \alpha)$, let $t\mapsto
V_t$ be the canonical representation of $\R^n$ in
$M(A\rtimes_\alpha \R^n)$, the multiplier algebra, with
$V_{x}aV_{x}^*=\alpha_{x}(a)$ ($a\in A$). Consider $B=A\otimes \C_n$ with the extended action $\alpha$ by
$\R^n$, which acts trivially on the component $\C_n$. The
smooth subalgebra $B^\infty$ of $\alpha$ is identified with
$A^\infty\otimes \C_n$. As $\Sigma\in S^0(\R^n, B^\infty)$ (as a symbol independent of $A^{\infty}$)
is a symbol of order $0$,
\[
\widehat{\Sigma}(x)=\int_{{\R^n}^*}\Sigma(\xi)e(-\langle x, \xi  \rangle)d\xi
\]
is a well-defined distribution on $\R^n$ with value in
$B^\infty$.  Using Connes' pseudo-differential calculus \cite{connes80} (c.f. Appendix \ref{sec:calculus} ) for $(A, \R^n, \alpha)$,  we define
$D^\infty_\alpha\in M(B^\infty\rtimes \mathbb{R}^n)$ by
\[
D^\infty_\alpha:=\int \widehat{\Sigma}(x)V_x dx.
\]

\begin{lemma}\label{prop:Sigma} The symbol function $\frac{\partial \Sigma}{\partial \xi_j}$, for $j=1,...,n$, is a Schwartz function on $\mathbb{R}^n$. 
\end{lemma}
\begin{proof}
The proof of this property is presented in Appendix \ref{app:schwarz}.
\end{proof}

The grading operator on $M(B^\infty\rtimes \mathbb{R}^n)$ is defined by the inner automorphism with respect to the element $\epsilon\in M(B^\infty\rtimes \mathbb{R}^n)$, where $\epsilon$ is the grading operator on the spinors associated to $\mathbb{C}_n$. It is straightforward to check that $\epsilon$ is invariant with respect to the $G$-action. 

\begin{lemma}\label{lem:equivariant-symbol}
The Fourier transform map which sends $\Sigma $ to $\widehat{\Sigma}$ is $G$-equivariant, i.e $\widehat{g\cdot \Sigma} = g\cdot \widehat{\Sigma}.$ 
	\end{lemma}
\begin{proof}
	\begin{align*}
\widehat{g\cdot \Sigma}(x)&= \int_{\R_n}(g\cdot \Sigma)(\xi)e(-\langle x, \xi  \rangle)d\xi		
\\ &= \int_{\R_n} \beta_g(\Sigma(g^t\xi))e(-\langle x, \xi  \rangle)d\xi		
\\ &= \int_{\R_n} \beta_g(\Sigma(\xi))e(-\langle x, (g^t)^{-1}\xi  \rangle)d\xi		
\\ &= \int_{\R_n}\beta_g(\Sigma(\xi))e(-\langle g^{-1} x, \xi  \rangle)d\xi		
\\ &= \beta_g(\widehat{\Sigma}(g^{-1} x))
\\ &= (g\cdot \widehat{\Sigma})(x).
\end{align*}
	\end{proof}

\begin{remark}
	The integrands of the above give rise to divergent integrals: we took the Fourier transformation of an order zero symbol (as tempered distribution). To regularise divergent oscillatory integrals, one does the following. 
	Since we realize $\widehat{\Sigma}$ as a distribution, for $u \in \mathcal{S}(\R^n,A^\infty)$, $\langle \widehat{\Sigma}, u \rangle$ is given by 
	
	$$ \lim_{\epsilon \rightarrow 0} \int_{\R^n}\int_{\R_n} \Sigma(\xi)u(x)e\left(- \frac{\epsilon\norm{\xi}^2}{2}\right) e(-\langle x, \xi  \rangle)d\xi dx.$$
	
	Now since the expression $e\left(- \frac{\epsilon\norm{\xi}^2}{2}\right)$ is $G$-invariant, the change of variable in the above proof makes sense. From now on we will use change of variables in oscillatory integrals for the action of $G$ without any further explanation.

\end{remark}

\begin{lemma} \label{lem:commutator}For any $a\in A^\infty$, $[a, D_\alpha^\infty]\in B^\infty\rtimes_\alpha \mathbb{R}^n$. 
\end{lemma}

\begin{proof}

Following the definition of $D_\alpha^\infty$, we compute the commutator $[a, D_\alpha^\infty]$ as follows. 

\begin{eqnarray*}
[a, D_\alpha^\infty]&=&[a, \int_{\mathbb{R}^n} \widehat{\Sigma}(x)V_xdx]\\
&=&\int_{\mathbb{R}^n}\int_{\mathbb{R}^n} [a, \chi\big(c(\xi)\big)e(-\langle x, \xi  \rangle)V_x d\xi dx]\\
&=&\int_{\mathbb{R}^n}\int_{\mathbb{R}^n}\chi\big( c(\xi)\big)e(-\langle x, \xi  \rangle)(a-\alpha_x(a))V_x d\xi dx\\
&=&\int_{\mathbb{R}^n}\int_{\mathbb{R}^n}\int_{\mathbb{R}} \widehat{\chi}(s)e^{isc(\xi)}e(-\langle x, \xi  \rangle) e(-\langle x, \xi  \rangle)(a-\alpha_x(a))V_x ds d\xi dx  .
\end{eqnarray*}

Define a function $F: \mathbb{R}^n\to B^\infty$ by $F(x):=\alpha_x(a)$. As $a$ belongs to $B^\infty$, $F(x)$ is a smooth  functions on $\mathbb{R}^n$. Therefore, we can write 
\[
F(x)-F(0)=\int_{0}^{1} \frac{d}{ds} F(sx) ds=\int_0^1 x^i \frac{\partial}{\partial x^i}F(sx) ds. 
\] 
Define $G_i: \mathbb{R}^n\to B^\infty$ by 
\[
G_i(x):=\int_0^1 \frac{\partial}{\partial x^i}F(sx) ds. 
\]
It is not hard to check that $G_i$ is again a smooth function on $\mathbb{R}^n$ with value in $B^\infty$.  

In summary, there are smooth  functions $G_i$, $i=1,..., n$, such that the following equation holds, 
\[
\widehat{\Sigma}(x)\big(\alpha_x (a)-a\big)=\widehat{\Sigma}(x)\sum_i x^i G_i(x)=\sum_i \widehat{\Sigma}(x)x^i G_i(x). 
\]

According to the Fourier transform formula, we have 
\[
\widehat{\Sigma}(x)x^i=\frac{1}{\sqrt{-1}}\widehat{\frac{\partial \Sigma}{\partial \xi_i}}(x). 
\]

Recall that $\Sigma(\xi)$ is defined as
\[
\Sigma(\xi):=1\otimes \chi\big(c(\xi)\big). 
\]

By Lemma \ref{prop:Sigma}, $\frac{\partial \Sigma}{\partial \xi_i}$ is a Schwartz function, and therefore $\widehat{\Sigma}(x)x^i$ is a Schwartz function. From the above discussion, we conclude that as a sum of elements in $B^\infty\rtimes_\alpha \mathbb{R}^n$, 
\[
\widehat{\Sigma}(x)\big(\alpha_x (a)-a\big)=\sum_i \widehat{\Sigma}(x)x^i G_i(x)
\]
belongs to $B^\infty\rtimes _\alpha\mathbb{R}^n$. 
\end{proof}

\begin{lemma}\label{smoothconnesthomcommutator2}
	$a(1-(D^\infty_{\alpha})^2) \in B^\infty \rtimes_\alpha \R^n$ for $a \in A^\infty \hookrightarrow M(B^\infty \rtimes_\alpha \R^n)$. 
	\end{lemma}
	
	\begin{proof}
		
	Noting that
		 $\hat{\Sigma} * \hat{\Sigma} $ is well-defined as a distribution, we have
		$$(1-(D^\infty_{\alpha})^2) = \int_{\R^n} (\widehat{1-\Sigma^2}(x))V_{x} dx.$$
		Now $$(1-\Sigma^2)(\xi) = 1 \otimes \big(1-\chi(c(\xi))^2\big).$$ 

As the function $s\widehat{\chi}(s)$ is a smooth function with compact support, the corresponding Fourier transform, $\frac{d \chi}{d\lambda}$, is a Schwartz function. Furthermore, $\chi^2(\lambda)\to 1$ as $\lambda \to \pm \infty$. It follows that $\chi^2(\lambda)-1$ is a Schwartz function. Therefore,  $(1-\Sigma^2)$ is a  Schwartz function and hence $-\infty$ order symbol. So the element $\int_{\R^n} (\widehat{1-\Sigma^2}(x))V_{x} dz	$ is in $B^\infty\rtimes_\alpha \R^n$ which proves our claim. 
	\end{proof}
	
\begin{lemma}\label{lem:g}
	$g\cdot D^\infty_{\alpha} = D^\infty_{\alpha}$.
\end{lemma}

\begin{proof}
We start with checking $\widehat{\Sigma}(gx) = g\cdot \widehat{\Sigma}(x).$ Indeed
\begin{align*}
\widehat{\Sigma}(gx) &=\int_{{\R_n}}\Sigma(\xi)e(-\langle gx, \xi  \rangle)d\xi
\\&=\int_{{\R_n}}\Sigma(\xi)e(-\langle x, g^t\xi  \rangle)d\xi
\\&=\int_{{\R_n}}\Sigma((g^t)^{-1}\xi)e(-\langle x, \xi  \rangle)d\xi
\\&=\int_{{\R_n}}\Sigma(g\cdot \xi)e(-\langle x, \xi  \rangle)d\xi
\\&=\int_{{\R_n}}\beta_g(\Sigma(\xi))e(-\langle x, \xi  \rangle)d\xi
\\&=g\cdot \widehat{\Sigma}(x).
\end{align*}

Using the above property, we verify the following equality.
\begin{align*}
g\cdot D^\infty_{\alpha} &=\int_{\R^n} g\cdot \widehat{\Sigma}(x)V_{gx} dx
\\ &= \int_{\R^n} \widehat{\Sigma}(gx)V_{gx} dx 
\\ &= \int_{\R^n} \widehat{\Sigma}(x)V_x dx 
\\ &= D^\infty_{\alpha}.
\end{align*}
\end{proof}

As the above locally convex Kasparov module $\big(M(B^\infty\rtimes \mathbb{R}^n), \iota, D^\infty_\alpha\big)$ is $G$-invariant, we introduce the following Kasparov module for the corresponding crossed product algebras. Consider $M\big((B^\infty\rtimes _\alpha \mathbb{R}^n)\rtimes G\big)$ with the natural inclusion map $\iota: A^\infty\rtimes G\to M\big((B^\infty \rtimes_\alpha \mathbb{R}^n)\rtimes G\big)$. The grading operator $\epsilon$ on $M(B^\infty\rtimes \mathbb{R}^n)$ extends to a grading operator on $M\big((B^\infty\rtimes \mathbb{R}^n)\rtimes G\big)$ defined by
\[
\tilde{\epsilon}(f)(g)=\epsilon\big(f(g)\big)),\ \forall g\in G.
\]

The operator $D^\infty_{\alpha}$ extends to an element $\widetilde{D}^\infty_{\alpha}$ in $M\big((B^\infty\rtimes \mathbb{R}^n)\rtimes G\big)$ by
\[
\widetilde{D}^\infty_\alpha(f)(g)=D^\infty_\alpha(f(g)),\ \forall g\in G.
\] 

With the above preparation, we directly conclude the following proposition by Lemma \ref{lem:commutator}, \ref{smoothconnesthomcommutator2}, \ref{lem:g}. 
\begin{proposition}\label{prop:Dirac}
The multiplier $\widetilde{D}^\infty_\alpha$ of the algebra $(B^{\infty}\rtimes \R^n)\rtimes G$ defines a locally convex Kasparov module $(M\big((B^{\infty}\rtimes \R^n)\rtimes G\big), \iota, \widetilde{D}^\infty_{\alpha} )$ in $AKM\big(A^{\infty}\rtimes G, (B^{\infty}\rtimes_{\alpha}\R^n)\rtimes G\big)$, which we will denote by $y^\infty_{n, \alpha}$. 
\end{proposition}

Applying Proposition \ref{prop:Dirac} to the dual action  $\hat{\alpha}$ of $\mathbb{R}_n$ on $(B^{\infty}\otimes \C_n) \rtimes_{\alpha}\R^n$, we obtain an element 
\[
\begin{split}
\Big(M\Big(\big(((B^{\infty}\otimes \C_n) \rtimes_{\alpha} \R^n)&\rtimes_{\hat{\alpha}}\mathbb{R}_n\big)\rtimes G\Big), \iota, D^\infty_{\hat{\alpha}} \Big) \\
&\in 
AKM\Big((B^{\infty}\rtimes_{\alpha}\R^n)\rtimes G, \big(((B^{\infty}\otimes \C_n)\rtimes_{\alpha}\R^n) \rtimes_{\hat{\alpha}} \mathbb{R}_n\big)\rtimes G \Big). 
\end{split}
\]
Using the $G$-equivariant Takesaki--Takai duality (Theorem \ref{thm:takai}), we have the isomorphism
\[
 \big(((B^{\infty}\rtimes \C_n) \rtimes_{\alpha}\R^n) \rtimes_{\hat{\alpha}} \mathbb{R}_n\big)\rtimes G\cong (A  \otimes \C_{2n} \otimes  \mathcal{K}^\infty)\rtimes G. 
\]
With the above isomorphism, we get a Dirac element 
\[
x^{{\infty}}_{n,\alpha} \in AKM\big((B^{\infty}\rtimes_{\alpha}\R^n)\rtimes G,(A^{\infty}\otimes \C_{2n} \otimes \mathcal{K}^\infty)\rtimes G\big).
\]

In summary, we have constructed in this section a (dual) Dirac element for an $\R^n$-action, i.e. 
\[
\begin{split}
x^{{\infty}}_{n,\alpha} &\in AKM\big((B^{\infty}\rtimes_{\alpha}\R^n)\rtimes G, (A^{\infty}\otimes \C_{2n} \otimes \mathcal{K}^\infty)\rtimes G\big),\\
y^{{\infty}}_{n,\alpha} &\in AKM\big(A^{\infty}\rtimes G, (B^{\infty}\rtimes_{\alpha}\R^n)\rtimes G\big).
\end{split}
\]
We next prove $\HH(y^{{\infty}}_{n,\alpha})$ and
$\HH(x^{{\infty}}_{n,\alpha})$ are inverses to each other for a split exact, diffotopy invariant, $\mathcal{K}^\infty$ stable functor $\HH$. 

\subsection{Equivariant Bott periodicity}\label{subsec:bott}
When  $\alpha$ is  trivial, we call the elements $y^{{\infty}}_{n,\alpha}$ and $x^{{\infty}}_{n,\alpha}$ by $y^{{\infty}}_{n}$ and  $x^{{\infty}}_{n}$ respectively. If we take the extension of the elements $y^{{\infty}}_{n}$ and  $x^{{\infty}}_{n}$  to the corresponding C*-algebras to get elements in $\KK\Big(\big(B\otimes C_0(\R^n)\big)\rtimes G,A\rtimes G\Big)$ and $\KK\Big(A\rtimes G,\big(B\otimes C_0(\R^n)\big)\rtimes G\Big)$, we get the usual class of Kasparov Dirac-dual Dirac element, since the usual class and the extended class are compact perturbation to each other.

\begin{remark}
	Grensing \cite{Grensing12} defined similar elements. Our $y^{{\infty}}_{n,\alpha}$ and $x^{{\infty}}_{n,\alpha}$  are slightly different from his. But as we want to use  $\mathcal{K}^\infty$-stability rather than stability with the Schatten ideals, we prefer to work with this modified class. Grensing's class and our class match up to compact perturbation if we view the classes as KK elements (with the obvious extensions) in the corresponding C*-algebras.
\end{remark}

\begin{theorem}\label{bott-periodicity}
	 Let $\HH$ be a split-exact, diffotopy invariant, $\mathcal{K}^\infty$ stable functor. Then for the trivial action $\alpha,$  $\HH(x^{{\infty}}_{n,\alpha})$ and $\HH(y^{{\infty}}_{n,\alpha})$ are inverse to one another.
\end{theorem}

\begin{proof}

We notice that the elements in 
\[
\KK\Big(\big(B\otimes C_0(\R^n)\big)\rtimes G,A\rtimes G\Big)\ \text{and}\ \KK\Big(A\rtimes G,\big(B\otimes C_0(\R^n)\big)\rtimes G\Big)
\]
are the usual Kasparov Dirac and dual Dirac elements. Now let $y$ and $x$ are the quasihomomorphisms associated to $y^{{\infty}}_{n}$ and  $x^{{\infty}}_{n}$ respectively. We wish to calculate $\HH(y) \circ \HH(x)$. Now let $\gamma$ be the product of these two elements whose index is 1 (using Kasparov's Bott periodicity \cite{MR1388299}). And therefore the element $\gamma$  has a trivial $G$ index. And the equation $\HH(y^{{\infty}}_{n}) \circ \HH(x^{{\infty}}_{n}) = 1$ follows from the second part of the proof of \cite[Theorem 57]{Grensing12} using  the standard ($G$-equivariant) rotation argument, Proposition \ref{prop:G-index}, and $\mathcal{L}^p$ replaced by $\mathcal{K}^\infty$ . 

Now using the arguments similar to Grensing's (the first part of \cite[Theorem 57]{Grensing12}) and Proposition \ref{prop:G-index}, we can similarly prove $\HH(x^{{\infty}}_{n}) \circ \HH(y^{{\infty}}_{n}) = 1. $

\end{proof}

As a corollary of Theorem \ref{bott-periodicity}, when $A=\mathbb{C}$, we have the following version of Bott-periodicity. 

\begin{corollary}\label{cor:bott}Let $\HH$ be a split-exact, diffotopy invariant, $\mathcal{K}^\infty$-stable functor. Then for the trivial $\alpha,$  $\HH(x^{{\infty}}_{n,\alpha})$ and $\HH(y^{{\infty}}_{n,\alpha})$ establish an isomorphism between $\HH(\mathbb{C}G)$ and $\HH\Big(\big(\C_n\otimes \mathcal{S}(\mathbb{R}^n)\big)\rtimes G\Big)$.
\end{corollary}
\subsection{Proof of the equivariant Connes--Thom isomorphism}\label{subsec:proof-Thom}
With the development in Section \ref{subsec:dirac} and \ref{subsec:bott}, we are ready to prove Theorem \ref{thm:equ-thom}. 

\begin{proof}
Define a deformation $\alpha^s$ of the action $\alpha$ for $s\in
[0,1]$. Define $\alpha^s:\R^n\times A^\infty \to A^\infty $ by
\[
\alpha^s_x(a):=\alpha_{sx}(a),\qquad a\in A, x\in \R^n.
\]
Then the same formulas as above define 
\[
\begin{split}
y^{{\infty}}_{n,\alpha^s}&\in AKM\Big(A^{\infty}\rtimes G, \big(B^{\infty}\rtimes_{\alpha^s}\R^n\big)\rtimes G\Big),\\
x^{{\infty}}_{n,\alpha^s}&\in AKM\Big(\big(B^{\infty}\rtimes_{\alpha^s} \R^n\big)\rtimes G, \big(A^{\infty}\otimes \C_{2n} \otimes \mathcal{K}^{\infty}\big)\rtimes G\Big). 
\end{split}
\]
Now let $y_s:=y^{{\infty}}_{n,\alpha^s} = (\beta_s,\bar{\beta}_s)$ and $x_s:=x^{{\infty}}_{n,\alpha^s} = (\hat{\beta}_s,\bar{\hat{\beta}}_s)$. Denote the corresponding quasihomomorphisms also by $y_s$ and $x_s$ respectively.

Let us now consider the algebra $D^\infty =\Zyl A^\infty$, and the $\R^n$ action on the algebra $D^\infty$ by $\gamma_s(f)(x) = \alpha_{xs}(f(x)).$ Now we have the following elements:
\begin{enumerate}
\item $$\HH(\ev_s) \in QH \big(D^\infty\rtimes G, A^\infty\rtimes G\big)\quad \text{(since $\ev_s$ is $G$-equivariant)},$$
\item $$ \HH(\widehat{\ev_s}) \in QH\Big(\big((D^\infty\otimes \C_n) \rtimes \R^n\big)\rtimes G ,\big((A^\infty \otimes \C_n) \rtimes \R^n\big)\rtimes G \Big)\quad \text{$(\widehat{\ev_s}$ is extended from $\\ev_s$)},$$
\item $$\HH(\widehat{\widehat{\ev_s}}) = \HH(\ev_s \otimes \id \otimes \id  ) \in QH((D^\infty\otimes \C_{2n} \otimes \mathcal{K}^\infty)\rtimes G, (A^\infty \otimes \C_{2n} \otimes \mathcal{K}^\infty)\rtimes G)$$ (using the Takesaki--Takai isomorphism (Theorem \ref{thm:takai}) and viewing $\ev_s \otimes \id \otimes \id$ as a $G$-equivariant map from $D^\infty\otimes \C_{2n} \otimes \mathcal{K}^\infty$ to $A^\infty \otimes \C_{2n} \otimes \mathcal{K}^\infty$),
\end{enumerate}
together with 
\[
\begin{split}
y^\infty&:= y_{n,\gamma}^\infty = (\gamma, \overline{\gamma})\in QH\Big(D^\infty\rtimes G,\big((D^\infty \otimes \C_n) \rtimes \R^n\big)\rtimes G\Big),\\
x^\infty&:= x_{n,\gamma}^\infty = (\hat{\gamma}, \overline{\hat{\gamma}})\in QH\Big(\big((D^\infty \otimes \C_n) \rtimes \R^n\big)\rtimes G,(D^\infty \otimes \C_{2n} \otimes \mathcal{K}^\infty)\rtimes G \Big),
\end{split}
\]
which satisfy 

$$\HH(\ev_s)\circ \HH(y_s) = \HH(y^\infty) \circ \HH(\widehat{\ev_s}),\ \text{and}\ \HH(x^\infty)\circ \HH(\ev_s) = \HH(\widehat{\ev_s}) \circ \HH(x_s).$$

The above two equations come from the following  diagrams and Proposition \ref{morphofdoublesplit}:

\begin{equation}\label{dia1}\xymatrix{ 0\ar[r]&{\big((A^\infty \otimes \C_n) \rtimes \R^n\big)\rtimes G}\ar[r] &\mathcal{D}_{\hat{\beta}_s}\rtimes G \ar[r]&{(A^\infty \otimes \C_{2n} \otimes \mathcal{K}^\infty)\rtimes G }\ar[r]&0\\
0\ar[r]&{\big((D^\infty \otimes \C_n) \rtimes \R^n\big)\rtimes G}\ar[u]_{\widehat{\ev_s}}\ar[r]^{}&\mathcal{D}_{\hat{\gamma}}\rtimes G\ar[u]_{\overline{\ev_s}}\ar[r]&{(D^\infty \otimes \C_{2n}\otimes \mathcal{K}^\infty)\rtimes G }\ar[u]_{\ev_s \otimes \id \otimes \id } \ar[r]&0
}\end{equation}
and 
\begin{equation}\label{dia2}\xymatrix{ 0\ar[r]&{A^\infty\rtimes G}\ar[r] &\mathcal{D}_{{\beta}_s}\rtimes G \ar[r]&{\big((A^\infty \otimes \C_n )\rtimes \R^n\big)\rtimes G}\ar[r]&0\\
0\ar[r]&{D^\infty\rtimes G}\ar[u]_{\ev_s}\ar[r]^{}&\mathcal{D} _{\gamma}\rtimes G\ar[u]_{\overline{\ev_s}}\ar[r]&{\big((D^\infty \otimes \C_n) \rtimes \R^n\big)\rtimes G}\ar[u]_{\widehat{\ev_s}} \ar[r]&0
}\end{equation}

The above two diagrams are reinterpretations of the naturality of the Thom elements. The diagrams commute and also commute with the double-splits (as in Proposition \ref{morphofdoublesplit}). Note that using $\mathcal{K}^\infty$-stability of $\HH,$  $\HH(\ev_s \otimes \id \otimes \id ) = \HH(\ev_s).$

Now

 	  \begin{align*}
 		 (\HH(y^\infty)\circ \HH(x^\infty)) \circ \HH(\ev_s)   &=  \HH(y^\infty)\circ (\HH(x^\infty) \circ \HH(\ev_s))
 		 \\&= \HH(y^\infty)\circ (\HH(\widehat{\ev_s}) \circ \HH(x_s))
 		\\&= (\HH(y^\infty)\circ \HH(\widehat{\ev_s})) \circ \HH(x_s)
 		  \\&= \HH(\ev_s)\circ \HH(y_s) \circ \HH(x_s).
 	\end{align*}

But using diffotopy invariance of  the $\HH$ functor, we have $\HH(\ev_s)= \HH(\ev_0).$ So   
 		$$ \HH(y_s)\circ \HH(x_s) =  \HH(\ev_0)^{-1} \circ (\HH(y^\infty) \circ \HH(x^\infty))\circ \HH(\ev_0)$$ shows that $ \HH(y_s)\circ \HH(x_s)$ is independent of $s$. Similarly starting with  
$ \HH(x^\infty)\circ \HH(y^\infty) \circ \HH(\widehat{\ev_s}),$ we conclude that  	$ \HH(x_s)\circ \HH(y_s)$ is independent of $s.$ This implies that 

$$\HH(x^{{\infty}}_{n,\alpha})\circ
\HH(y^{{\infty}}_{n,\alpha})= \HH(x^{{\infty}}_{n,\alpha^0})\circ
\HH(y^{{\infty}}_{n,\alpha^0}).$$
The product $\HH(x^{{\infty}}_{n,\alpha})\circ
\HH(y^{{\infty}}_{n,\alpha})$ is reduced to the computation to the Bott-periodicity result, Theorem \ref{bott-periodicity}:

\[
\HH(x^{{\infty}}_{n,\alpha^0})\circ
\HH(y^{{\infty}}_{n,\alpha^0})=1,
\]
where the action of $\R^n$ on $A^\infty$ is trivial. We conclude that
$\HH(x^{{\infty}}_{n,\alpha}) \circ
\HH(y^{{\infty}}_{n,\alpha})$ is equal to $1$ in $QH\big(A^\infty\rtimes G, (A^\infty \otimes \mathcal{K}^\infty)\rtimes G\big).$ And a similar computation shows that
$\HH(y^{{\infty}}_{n,\alpha}) \circ
\HH(x^{{\infty}}_{n,\alpha})$ is 1 in $QH\big((B^\infty\rtimes_{\alpha}\R^n)\rtimes G, (B^\infty \rtimes_{\alpha}\R ^n)\rtimes G\big)$.
\end{proof}

We have the following corollary of Theorem \ref{thm:equ-thom}.  

\begin{corollary}\label{cor:connes-thom-cyclic} With all the notations and conditions as in Theorem \ref{thm:equ-thom}, when $\HH$ is  the periodic cyclic homology functor $\HP_\bullet$, we have
\[
\HP_\bullet(((A^{\infty}\otimes\C_n)\rtimes_\alpha
\R^n)\rtimes G)\cong \HP_\bullet(A^{\infty}\rtimes_{\beta}G ).
\]

\end{corollary}

\begin{proof} 
We observe that  periodic cyclic homology, $\HP(-)$, has all the properties which are assumed for the functor $\HH$ in Theorem \ref{thm:equ-thom}, c.f. \cite{MR823176}.  
With the $\mathcal{K}^\infty$-invariance of $\HP$ in the category of locally convex algebra, the result follows from Theorem \ref{thm:equ-thom}. 
\end{proof}


\section{Application to the Rieffel strict deformation quantization}\label{sec:quantization}
In this section, we apply the equivariant Connes--Thom isomorphism Theorem \ref{thm:equ-thom} for locally convex algebras to study cyclic homology of Rieffel's strict deformation quantization. 
Recall that given a strongly continuous action $\alpha$ of $\R^n$ on a $C^*$-algebra $A$, and a skew-symmetric form $J$ on
$\R^n$, Rieffel \cite{Rieffel93} constructed a strict deformation quantization $A_J$ of $A$ via oscillatory integrals, 
\begin{equation}\label{eq:star}
a\times_J b:=\int_{\R^n\times \R^n}
\alpha_{Jx}(a)\alpha_{y}(b)e( x\cdot y) dx dy,
\end{equation}
for $x,y\in \R^n$, and $a,b\in A^\infty$. The first copy of $\R^n$ in $\R^n\times \R^n$ is basically $\R_n$ after the identification of $\R_n$ and $\R^n$ (see the discussion at Page 11 of \cite[Chapter 2]{Rieffel93}.) $\R^n$ acts on $A_J$ by the same action $\alpha$ (\cite[Proposition 2.5]{Rieffel93}), and we denote the smooth vectors for this action by $A^\infty_J.$

We recall the following result about Rieffel's strict deformation quantization from \cite{Nesh14}:
\begin{proposition}\label{nesh:map}
The map $\Theta _J$ from $A^\infty_J \rtimes \R^n$	to $A^\infty \rtimes \R^n$	defined by 
$$
\Theta _J(f)(x) = \int_{\R^n} \alpha_{Jy}(\hat{f}(y))e(x\cdot y)dy
$$
is an isomorphism, where $\hat{f}$ is the Fourier transformation of $f \in \mathcal{S}(\R^n, A)$ and $e(t) := e^{2\pi i t} $.

\end{proposition}

\begin{proof}
	See \cite[Theorem 1.1]{Nesh14}.
\end{proof}

Let $G$ be a finite group acting on $A$ as in Section \ref{subsec:main-setup}. This means that $G$  acts  on $A$ by $\beta$ satisfying
\[
\beta_g\alpha_x=\alpha_{\rho_g(x)}\beta_g,\ \ \ \text{for any}\ g\in
G, x\in \R^n.
\]

We compute
\begin{align*}
\beta_g(a)\times_J \beta_g(b)&= \int_{\R^n\times \R^n}
\alpha_{Jx}(\beta_g(a))\alpha_{y}(\beta_g(b))e( x\cdot y) dx dy				
\\ &= \int_{\R^n\times \R^n}
\alpha_{(g^t)^{-1}Jg^{-1}x}\beta_g(a)\alpha_{gy} \beta_g(b)e(x\cdot gy) dx dy
\\ &= \int_{\R^n\times \R^n}
\alpha_{(g^t)^{-1}Jx}\beta_g(a)\alpha_{gy} \beta_g(b)e( gx\cdot gy) dx dy
\\ &= \int_{\R^n\times \R^n}
\beta_g \alpha_{Jx}(a)\beta_g \alpha_{y}(b)e( x\cdot y) dx dy
\\ &= \beta_g(a \times_J b).
\end{align*}
The above computation shows that the action of $G$ on $A_J$ by $\beta$ is well defined. So we get a $G$ action on $A^\infty_J \rtimes \R^n$ and $A^\infty \rtimes \R^n$. Abusively we call both actions by $\beta$ again.

The following property ensures that we can make the isomorphism in Proposition \ref{nesh:map} $G$-equivariant.

\begin{proposition}\label{chak:equivariant}
With the notations introduced in Proposition \ref{nesh:map},
$$
\beta_g(\Theta_J(f))= \Theta_J(\beta_g(f))
$$
\end{proposition}
\begin{proof}

\begin{align*}
\Theta_J(\beta_g(f))(x) &= \int_{\R^n} \alpha_{Jy}(\widehat{\beta_g f}(y))e(x\cdot y)dy				
\\ &= \int_{\R^n} \alpha_{Jy}(\int_{\R^n} \beta_g f(t)e(-y\cdot t)dt))e(x\cdot y)dy
\\
&= \int_{\R^n}\int_{\R^n} \alpha_{Jy} \beta_g (f(g^{-1}t))e(-y\cdot t)e(x\cdot y)dtdy
\\
&= \int_{\R^n}\int_{\R^n} \beta_g \alpha_{g^{-1}Jy}  (f(g^{-1}t))e(-y\cdot t)e(x\cdot y)dtdy
\\
&= \int_{\R^n}\int_{\R^n} \beta_g \alpha_{g^{-1}Jy}  (f(t))e(-y\cdot gt)e(x\cdot y)dtdy
\\
&= \int_{\R^n}\int_{\R^n} \beta_g \alpha_{g^{-1}J(g^t)^{-1}y}  (f(t))e(-y\cdot t)e(g^{-1}x\cdot y)dtdy
\\
&= \int_{\R^n}\int_{\R^n} \beta_g \alpha_{Jy}  (f(t))e(-y\cdot t)e(g^{-1}x\cdot y)dtdy
\\
&= \beta_g(\Theta_J(f))(x),
\end{align*}
where in the above formulas we have used $gx$ to denote $\rho_g(x)$.
\end{proof}

\begin{rem}

	In the above, since $y \in \R_n$,  abusively we have written $gy$ for $(g^t)^{-1}y$  and we have used the fact that $g^{-1}J(g^t)^{-1} = J.$ In general one should be careful with $\R_n$ and $\R^n.$

\end{rem}

\begin{example}\label{ex:NC}

Recall that an $n$-dimensional noncommutative torus $A_{\theta}$ is
the universal C*-algebra  generated by unitaries $U_1$, $U_2$,  $U_3$, $\cdots$, $U_n$
subject to the relations
\[
U_k U_j = \exp (2 \pi i \theta_{jk} ) U_j U_k
\]
for $j, k = 1, 2, 3, \cdots, n$ and $\theta =: (\theta_{jk})$ being a skew symmetric real $n \times n$ matrix. 
If we look at the smooth holomorphically closed subalgebra $A^\infty_{\theta}$ of $A_{\theta}$. 
The algebra $A^\infty_{\theta}$ can also be viewed as the Rieffel strict deformation quantization \cite{Rieffel93} of $C^{\infty}(\mathbb{T}^n )$ by the translation action of $\R^n$ on $C^\infty(\mathbb{T}^n)$ and $\theta$ a skew symmetric form on $\mathbb{R}^n$. 

\end{example}
\begin{example}\label{ex:G}

Let $G$ be  $\Z_2$ or $\Z_3$ or $\Z_4$ or $\Z_6$ as finite cyclic groups (can be viewed as matrices in $SL_2(\Z)$) acting on $\R^2$. Since the action is $\Z^2$ preserving, the $2$-torus ${\mathbb
T}^{2}=\R^{2}/\Z^{2}$ inherits an action of
$G$ from the $G$ action on $\R^{2}$.  Let $A$ be the $C^*$-algebra
of continuous functions on ${\mathbb T}^{2}.$ The
group $\R^{2}$ acts on $\T^{2}$ by translation. For $\theta \in \R$, we consider the symplectic form $ \theta dx_1 \wedge dx_2$, also denoted by $\theta.$ So $A_\theta$ is just like the previous example of a 2 dimensional noncommutative torus. The action $\alpha$
(and $\beta$) of $\R^{2}$ (and $G$) on $A$ satisfy Eq. (\ref{eq:action}). Now the $G$ action on $A_\theta$ is well defined.

Recall that we can also consider the twisted group algebra  $C^*(\Z^2\rtimes G, \omega_{\theta})$, where $\omega_{\theta}$ is a 2-cocycle of $\Z^2$ ($\omega_{\theta}(x,y):= e^{2\pi i \langle \theta x,y\rangle}$, $\theta$ being a real number) extended trivially to the semi-direct product. These algebras are considered in \cite{Ech10}. Now it is not hard to see that  $C^*(\Z^2\rtimes G, \omega_{\theta}) = A_{\theta} \rtimes G,$ where the latter is defined as in the previous paragraph (see \cite[Lemma 2.10]{Ech10}).  In general, with the above 2-cocycles, the twisted group algebras of groups like $\Z^n\rtimes G$ are basically coming from equivariant strict deformation quantization of $\R^n$ action on $A = C(\T^n).$

\end{example}

\begin{corollary}\label{cor:ktheoryoforbifold}
	$\HP_\bullet(A_{\theta}^\infty \rtimes G)$ is independent of $\theta$ parameter. 
\end{corollary}
\begin{proof}

Here we have $A = C(\mathbb{T}^2)$ and $J=\theta.$ 
	From the above discussion and Proposition \ref{chak:equivariant} we get,
	
	$$(A_\theta^\infty \rtimes \R^2 )\rtimes G \simeq (A^\infty \rtimes \R^2) \rtimes G.$$
	
	Now applying the $\HP$ functor on the both sides we get,
	
	\begin{equation}
	\HP_\bullet\Big((A_\theta^\infty \rtimes \R^2 )\rtimes G\Big) = \HP_\bullet\Big(( A^\infty \rtimes \R^2) \rtimes G\Big).
	\end{equation}
	
	Now since in this particular case $G$ is a finite cyclic group, the $G$ action on $\R^2$ is $spin^c$ preserving. Indeed, the diagram 
	
	\begin{center}

\begin{tikzcd}
                &                                 &                                   & G                 \arrow{d}          &   \\
    0 \arrow{r} & \mathbb{S}^1 \arrow{r}   & Spin^c(n)                 \arrow{r}        & SO(n)           \arrow{r} & 0 \\
    
\end{tikzcd}
\end{center}

\noindent determines a group 2-cocycle on $SO(n)$. And since the restriction of this cocycle to $G$ is trivial (as $G$ is cyclic and preserving the symplectic structure $J$), the lift

	\begin{center}
		\begin{tikzcd}
                &                                 &                                   & G  \arrow[ld, dashrightarrow]               \arrow{d}          &   \\
    0 \arrow{r} & \mathbb{S}^1 \arrow{r}   & Spin^c(n)                 \arrow{r}        & SO(n)           \arrow{r} & 0 ,\\
    \end{tikzcd}
    \end{center}
   \noindent is always possible. Hence 
\begin{align*}
\HP_\bullet\Big((A_{\theta}^\infty \rtimes \R^2) \rtimes G\Big) 
&= \HP_\bullet \Big(\big( (A_{\theta}^\infty \otimes \C_2) \rtimes \R^2\big)\rtimes G\Big)
\\ &= \HP_\bullet(A_{\theta}^\infty\rtimes G)\ (\text{by Corollary \ref{cor:connes-thom-cyclic}}).
\end{align*}
Similar computation gives $\HP_\bullet( (A \rtimes \R^2) \rtimes G) = \HP_\bullet( A \rtimes G)$. So the claim follows from the isomorphism given by Equation (\ref{eq:cyclic}).

	\end{proof}

Now, from \cite[Theorem 0.1]{Ech10}, we know that  $\K_0(A_\theta\rtimes \Z_2) = \Z^6$, $\K_0(A_\theta\rtimes \Z_3) = \Z^8$, $\K_0(A_\theta\rtimes \Z_4) = \Z^9$, $\K_0(A_\theta\rtimes \Z_6) = \Z^{10}$ and $\K_1 = 0$ for all the cases. It is also well known that $A_\theta^\infty \rtimes G$ is holomorphically closed inside $A_\theta \rtimes G,$ for such G (see \cite[Proposition 6.6]{chak19}). Since the Chern character from $\K_*(C^\infty(\mathbb{T}^2)\rtimes G) \otimes \C$ to $\HP_*(C^\infty(\mathbb{T}^2)\rtimes G)$ is an isomorphism (by a result of Baum and Connes, see \cite[Page 279-80, equation 11 and 13]{Soll05}), we conclude that $\HP_0(A_\theta^\infty\rtimes \Z_2) = \C^6$, $\HP_0(A_\theta^\infty\rtimes \Z_3) = \C^8$, $\HP_0(A_\theta^\infty\rtimes \Z_4) = \C^9$, $\HP_0(A_\theta^\infty\rtimes \Z_6) = \C^{10}$ and $\HP_1 = 0$ for all these algebras.

\begin{corollary}\label{cor:cyclic-quantization}
	For a general deformation $A_{J}^\infty$ defined by Equation (\ref{eq:star}), $\HP_\bullet(A_{J}^{\infty}\rtimes G)$ is independent of the $J$ parameter. 
\end{corollary}
\begin{proof}
	We first prove the case that $J$ is nondegenerate. From Proposition \ref{nesh:map} we get,
	$$A^\infty_J \rtimes \R^n  \simeq A^\infty \rtimes \R^n .$$ 
	Again, the $G$ action on $\mathbb{R}^n$ preserves the matrix $J$ and therefore preserves a $spin^c$ structure. 
	Now applying the functor $HP_\bullet$ on the both sides we get the desired conclusion for $A^\infty_J$ with a skew symmetric $n\times n$ matrix $J$ by  a similar computation as in  the proof of Corollary \ref{cor:ktheoryoforbifold}. 
	
	For a general $J$, we decompose $\mathbb{R}^n$ into a direct sum of $V\oplus W$ for two $G$-invariant subspaces of $\mathbb{R}^n$ where the restriction of $J$ on $V$ is zero, and the restriction of $J$ on $W$ is nondegenerate. One directly checks that the algebra $A^\infty_J$ is isomorphic to $A^\infty_{J|_W}$ with the $W$-action on $A$ and $J|_W$. We can then deduce the desired result for general $A^\infty_J=A^\infty_{J|_W}$ from the above nondegenerate case.   
	\end{proof}

\appendix

\section{An equivariant version of Grensing's results}\label{sec:equi_Grensing}

Let $G$ be a finite group. We develop an equivariant version of Grensing's results \cite{Grensing12} and prove Proposition \ref{prop:G-index}.

The following result is a straight forward generalisation of Grensing's work \cite[Lemma 34]{Grensing12}. Since $G$ is a finite group, $\C G$ is isomorphic to $\oplus_{i=1}^d M_{n_i}(\C)$ with a trivial grading, where $d$ is the number of conjugacy classes of $G$, and for $i=1,...,d$, $\C^{n_i}$ is an irreducible $G$-representation. 
\begin{lemma}\label{lemma:induction_equivariant} Let $(\alpha,\bar\alpha):\C G\rightarrow\hat {\mathcal{B}}\trianglerighteq \mathcal{B}\rtimes G$ be a quasihomomorphism. Then there is a quasihomomorphism $(\alpha',\bar\alpha'):\C G \rightarrow M_{2d}({(\mathcal{B}\rtimes G)}^+)\trianglerighteq M_{2d}(\mathcal{B}\rtimes G)$ such that
$$\HH(\theta_\mathcal{B})\circ \HH(\alpha,\bar\alpha)\circ \HH(\kappa)^{-1} =\HH(\alpha',\bar\alpha')$$
for every split exact $\mathcal{K}^\infty$-stable functor, where $\theta_{\mathcal{B}}:\mathcal{B}\to M_{2d}(\mathcal{B})$ denotes the $G$-stabilisation, and $\kappa$ is the stabilisation map from $\C^d$ to $\oplus_{i=1}^d M_{n_i}(\C).$
\end{lemma}

\begin{proof}

For each summand of $\C G$, $M_{n_i}(\C),$ consider the quasihomomorphism  $\kappa$ from $\C$ to $M_{n_i}(\C)$ given by stabilisation. Now the composition of  $\kappa$ and  $(\alpha,\bar\alpha)$, is given by a quasihomomorphism (using \cite[Part (i) of Corollary 56]{Grensing12}) from $\C$ to $\mathcal{B}\rtimes G$. So by \cite[Lemma 34]{Grensing12}, this quasihomomorphism gives rise to a quasihomomorphism $(\alpha',\bar\alpha') : \C  \rightarrow M_{2}({(\mathcal{B}\rtimes G)}^+)\trianglerighteq M_{2}(\mathcal{B}\rtimes G)$ such that  
$$\HH(\theta)\circ \HH(\alpha,\bar\alpha)\circ \HH(\kappa)^{-1} =\HH(\alpha',\bar\alpha'),$$ where $\theta$ is the $M_2$ stabilisation. Now the result follows from taking the direct sum of each summand.

\end{proof}

As an immediate corollary we get the following generalisation of \cite[Part (i) of Corollary 56]{Grensing12}.

\begin{corollary}\label{sec:product_Grensing} Every two quasihomomorphisms $(\alpha,\bar\alpha):\C G \rightarrow \hat {\mathcal{B}} \trianglerighteq \mathcal{B}\rtimes G$ and $(\beta,\bar\beta):\mathcal{B}\rtimes G\rightarrow\hat C \trianglerighteq C$ 
have (up to $M_{2d}$-stabilisation) a product with respect to split exact $\mathcal{K}^\infty$-stable functors.
\end{corollary}
\begin{proof} $\mathcal{K}^\infty$-stability implies $M_{2d}$-stability. Now the same arguments of \cite[Part (i), Corollary 56]{Grensing12} and  Lemma \ref{lemma:induction_equivariant} give the result.\end{proof}

Now we define an $R(G)$ module structure of $\HH(\mathcal{A}\rtimes G)$ following the construction in \cite[Section 2.7, Definition 2.7.8] {MR911880} for a  split-exact, $\mathcal{K}^\infty$-stable functor $\HH$ on the category of (graded) locally convex algebras. For any element $\rho$ of $R(G),$ we define a morphism $[\rho]$ from $\HH(\mathcal{A}\rtimes G)$ to itself in the following way. 

Suppose that $\rho$ is a $G$ representation on a finite dimensional vector space $V$ of dimension $l$, and so it induces an action of $G$ on $L(V).$  Firstly, we have an algebra homomorphism $\varphi: \mathcal{A}\rtimes G\to \big(L(V)\rtimes \mathcal{A})\big)\rtimes G$ mapping $a\otimes f$ to $(id\otimes a)\otimes f$. Secondly,  the  crossed product $(L(V)\otimes \mathcal{A})\rtimes G$ is isomorphic to $(M_l(\C)\otimes \mathcal{A})\rtimes G$, where the $G$ action on $M_l(\C)$ is trivial.  The isomorphism $F$ from $(M_l(\C)\otimes \mathcal{A})\rtimes G$  to $(L(V)\otimes \mathcal{A})\rtimes G$ is given by $F(M\otimes f)(x)= M\rho_{x^{-1}}\otimes f(x).$  Thirdly, we observe that $(M_l(\C)\otimes \mathcal{A})\rtimes G$  is  isomorphic to $M_l(\mathcal{A}\rtimes G)$ via an isomorphism $\psi$. In summary, we have obtained a homomorphism $\Phi_V: \mathcal{A}\rtimes G\to M_l(\mathcal{A}\rtimes G)$. 

Let $\C$ be the trivial $G$ representation. We apply the above construction of $\Phi_V$ to the $G$ representation on the direct sum $V\oplus \mathbb{C}$. We obtain a homomorphism $\Phi_{V\oplus \mathbb{C}}$ from $\mathcal{A}\rtimes G$ to $M_{l+1}(\mathcal{A}\rtimes G)$. Let $p_V$ and $p_\C$ be the projections from $V\oplus \C$ onto $V$ and $\C$. We consider the homomorphisms $\varphi_V, \varphi_\C: \mathcal{A} \rtimes G\to (L(V\oplus \C)\otimes \mathcal{A})\rtimes G$ defined by
\[
\varphi_V(a\otimes f)=(p_V\otimes a)\otimes f,\ \varphi_\C(a\otimes f)=(p_\C\otimes a)\otimes f.
\]
By the $l+1$-stability, we have that $\HH(\Phi_V\circ \varphi_\C): \HH(\mathcal{A}\rtimes G)\to \HH\big(M_{l+1}(\mathcal{A}\rtimes G)\big)$ is an isomorphism. We define $[\rho]: \HH(\mathcal{A}\rtimes G)\to \HH(\mathcal{A}\rtimes G)$ to be 
\[
\HH(\Phi_V\circ \varphi_\C)^{-1}\circ \HH(\Phi_V\circ \varphi_V): \HH(\mathcal{A}\rtimes G)\to \HH(\mathcal{A}\rtimes G).
\]
It is straightforward to check that the above definition defines an $R(G)$ module structure on $\HH(\mathcal{A}\rtimes G)$, which is left to the reader. 
 
Now we have the following generalisation of \cite[Proposition 45]{Grensing12}.
 
 \begin{proposition}\label{prop:index} (Proposition \ref{prop:G-index})
 	Let $x=(\mathbb{B}(\mathcal{H}\rtimes G), \phi, F)$ be a locally convex Kasparov $(\C G , \mathcal{K}^\infty\rtimes G )$ module defined above from a $G$-equivariant Kasparov module $(\mathcal{H}, \varphi, \tilde{F})$ for $C^*$-algebras $(\C, \mathcal{K})$, and $\HH$ be a diffotopy invariant, split exact, $\mathcal{K}^\infty$ stable functor. Then  $\HH(Qh(x)) = \operatorname{index}_G \tilde{F} \circ \theta_*$, where $\theta_* : \HH(\C G) \rightarrow \HH(\mathcal{K}^\infty\rtimes G)$ denotes the $G$-stabilisation map, and $\operatorname{index}_G \tilde{F}$ is the $G$-index of the operator $\tilde{F}$ which is an element of $R(G).$
 \end{proposition}
 
 \begin{proof}
 	 
 	 The proof is essentially same to the proof of \cite[Proposition 45]{Grensing12}.
 Using Grensing's notation, we assume 	 
 	 $\tilde{F}=\left(\begin{array}{cc}0&S\\ T&0\end{array}\right)$. By  the hypothesis, $T$ is Fredholm and hence has closed co-kernel. The co-kernel is a $G$ space, and further assume that $1-TT'$ (using Grensing's notation) be a map from $\C^l \rtimes G$ to $(\C^l\otimes \mathcal{K}^\infty)\rtimes G $ (as a map of Hilbert modules), where $\C^l$ carries the $G$ action via an irreducible representation $\chi$ of $G$. If the action were  trivial we would get $\HH(1-TT')=-l\theta_*$ (using Grensing's non-equivariant version). We observe that since $G$ is finite, there always exists a $G$ invariant rank one projection in $\mathcal{K}^\infty.$ Then we can reduce the proof to the trivial case by applying the element $[\chi]$ on $\HH(\mathcal{K}^\infty \rtimes G).$    

\end{proof}


\section{Connes' pseudo-differential calculus}\label{sec:calculus}

We review briefly in this appendix key results in Connes' psuedodifferential calculus for $\mathbb{R}^n$ actions, \cite{connes80}.  We assume that the reader is familiar with the definition of classical pseudo-differential calculus of $\R^n$ i.e H\"ormander classes of symbols, and refer to \cite{Hoer65} for a thorough discussion of the classical theory. Connes in \cite{connes80} introduced an anisotropic version of H\"ormander classes of symbols, which was studied later by Baaj in \cites{baaj1,baaj2} in detail. 

Suppose that  $\mathbb{R}^n$, an abelian Lie group, acts (with the action denoted by $\alpha$) strongly continuously on a unital (adjoining a unit if necessary) C*-algebra $A$ (possibly graded). Let $A^{\infty}$ be the locally convex algebra of $A$ of smooth vectors for the action with a system of seminorms $(p_i)$. 
For the $C^*$-dynamical system $(A, \R^n, \alpha)$, let $x\mapsto
V_x$ be the canonical representation of $\R^n$ in
$M(A\rtimes_\alpha \R^n)$, the multiplier algebra, with
$V_{x}aV_{x}^*=\alpha_{x}(a)$ ($a\in A$).
 Let $\mathbb{R}_n$ be the Fourier dual of $\mathbb{R}^n$ as before.
 We shall say that $\rho$, a $C^{\infty}$ map from $\mathbb{R}_n$ to $A^{\infty}$, is a symbol of order $m$, $\rho \in S^m(\R_n,A^\infty)$ if the following properties hold:
\begin{enumerate}
\item for all multi-indices $i,j$, there exists $C_{ij} < 
\infty$ such that
$$
p_i \left(\left( \frac{\partial}{\partial \, \xi} \right)^j \, \rho 
(\xi) \right) \leq C_{ij} (1 + \vert \xi \vert)^{m-\vert j \vert} \, ;
$$
\item there exists $s \in C^{\infty} (\mathbb{R}_n \setminus \{ 0 \} , 
A^{\infty})$ such that when $\lambda \rightarrow + \infty$ one has 
$$\lambda^{-m} \, \rho (\lambda 
\, \xi) \rightarrow s (\xi)$$ 
$[$ for the topology of $C^{\infty} (\mathbb{R}_n \setminus \{ 0 
\} , A^{\infty})]$.
\end{enumerate}

When $A$ is $C_0(\R^n)$ and $\R^n$ is acting on $A$ by the translation action, we may think of $\rho$ as a two variable function. In this case we get back to the classical symbols (\cite[Lemma 2.7]{Lein10}).

Connes proved that an order zero symbol gives rise to an element of the multiplier algebra of the crossed product $A\rtimes_\alpha \R^n$. Indeed, if $\rho$ is a symbol of order zero then we can take the Fourier transform (in the sense of distribution): 
\[
\widehat{\rho}(x)=\int_{{\R_n}}\rho(\xi)e(-\langle x, \xi  \rangle)d\xi,
\] which
is a well-defined distribution on $\R^n$ with value in
$A^\infty$ (it will be clear later what a distribution means). Following \cite[Prop. 8]{connes80}, we define the multiplier of $A^\infty \rtimes \R^n$ which extends to an element
$D_\rho\in M(A\rtimes \R^n)$ by
\[
D_\rho:=\int_{\R^n} \widehat{\rho}(x)V_x dx.
\]
Following \cite[Definition 3.1]{baaj1}, { $D_\rho$ acts on the smooth sub-algebra $\mathcal{S}(\R^n,A^\infty)$ of $A\rtimes_\alpha \R^n$ by the oscillatory integral (see \cite[Section 3.3]{Abels12})
\[
D_\rho(u)(t):=\int_{\R^n}\int_{\R_n} \alpha_{-t}({\rho}(\xi))u(s)e(-\langle (t-s), \xi  \rangle)dsd\xi.
\]

To motivate the above equation, let us take $\rho \in \mathcal{S}(\R^n,A^\infty).$ Then  

	\begin{align*}
D_\rho(u)(t) &= \Big(\int_{\R^n} \widehat{\rho}(s)V_s u ds\Big)(t)
\\ &=  \int_{\R^n} \alpha_{-t}(\widehat{\rho}(s))V_s(u(t)) ds		
\\ &= \int_{\R^n} \alpha_{-t}(\widehat{\rho}(s))u(t-s) ds		
\\ &= \int_{\R^n} \alpha_{-t}(\widehat{\rho}(t-s))u(s) ds	
\\ &= \int_{\R^n}\int_{\R_n} \alpha_{-t}({\rho}(\xi))u(s)e(-\langle (t-s), \xi  \rangle)dsd\xi.
\end{align*}
\noindent Note that the above integrals exist in the usual sense.

The set (norm closure) of all multipliers, which come from order zero symbols, is denoted by $\mathcal{D}(A\rtimes \R^n).$ From \cite[Prop. 8]{connes80} and \cite{baaj1}  there is an exact sequence

\[
\begin{CD}
0 @>{}>>A\rtimes_\alpha \R^n@>{\phi}>> \mathcal{D}(A\rtimes_\alpha \R^n) @>{\psi}>>A\otimes C(S^{n-1}) @>{}>>0
\end{CD}.
\]
This exact sequence is often called the pseudo-differential extension.
It is well known that there is a non-degenerate morphism $C^*(\R^n)$ to  $M(A\rtimes_\alpha \R^n)$. So this morphism extends to the multiplier algebra of $C^*(\R^n)$ and in particular to the sub-algebra $\mathcal{D}(C^*(\R^n))$. So if we say $D \in \mathcal{D}(C^*(\R^n)),$ we view $D$ inside  $M(A\rtimes_\alpha \R^n).$

\begin{theorem}\label{pseudopropeerties2}
	For $a \in A \hookrightarrow M(A\rtimes \R^n)$ and $D \in \mathcal{D}(C^*(\R^n))$ , we have $[D,a] \in A\rtimes_\alpha \R^n.$ 
\end{theorem}
\begin{proof}
	See \cite[Section 4]{baaj1}, also \cite[Proposition 4.3]{Skan15} for a more general version of this property.
\end{proof}

Let us recall the definition of asymptotic expansion of a symbol. For a decreasing divergent  sequence $(m_j)_{j \in \{0,1,2,\cdots \}}$, and $\rho_j \in S^{m_j}(\R_n, A^\infty)$, we say $\rho \in S^{m_0}(\R_n, A^\infty)$ admits an asymptotic expansion $\sum\rho_j$ (written as $\rho \sim \sum\rho_j$), if for all integers $k \ge 1,$ 
$$\rho - \sum_{j \le k}\rho_j \in S^{m_k}(\R_n, A^\infty).$$ 
For a multi-index $k$, and $a \in A^\infty$, let us denote the $k$-th derivative of $a$ (with respect to the action of $\R^n$) by $\delta^k(a).$
\begin{theorem}\label{pseudopropeerties3}
	For $\rho_1 \in S^{m_1}(\R_n, A^\infty)$ and  $\rho_2 \in S^{m_2}(\R_n, A^\infty)$, there exists a unique $\rho \in S^{m_1+m_2}(\R_n, A^\infty)$  such that $D_\rho = D_{\rho_1}D_{\rho_2}.$ Also $\rho$ admits an asymptotic expansion:
	
	$$\rho(\xi) \sim \sum_k \frac{ i^{\absv{k}}}{k!}\rho_{1}^{(k)}(\xi)\delta^k(\rho_2(\xi)). $$ 
	
\end{theorem}

\begin{proof}
	See \cite[Proposition 3.2]{baaj1}. Also \cite[Theorem 2.2]{Lesch16}, for twisted dynamical systems.
\end{proof}

\begin{theorem}\label{pseudopropeerties4}
		For $\rho \in S^{0}(\R_n, A^\infty)$ the the adjoint of $D_\rho$, $(D_\rho)^*$ exists and  
		$(D_\rho)^* = D_\rho',$ where $\rho'$ admits an asymptotic expansion:
	
	$$\rho'(\xi) \sim \sum_k \frac{ i^{\absv{k}}}{k!}\delta^k((\rho')^{(k)}(\xi)^*). $$ 
\end{theorem}
\begin{proof}
	See \cite[Proposition 3.3]{baaj1}. Also \cite[Theorem 2.2]{Lesch16}, for twisted dynamical systems.
\end{proof}

\begin{remark}
Since the unitisation of $A$, $A'$, sits inside $M(A)$ non-degenerately, we get (\cite[Proposition 3.2]{Skan15})  a non-degenerate morphism from $A' \rtimes \R^n$ to $M(A\rtimes \R^n)$ giving a morphism from $\mathcal{D}(A'\rtimes_\alpha \R^n)$ to $M(A\rtimes \R^n)$. Hence, though we adjoin a unit for the non-unital $A$, ultimately we end up with getting an element in $M(A\rtimes \R^n).$  
\end{remark}

\section{Proof of Lemma \ref{prop:Sigma}}
\label{app:schwarz}
In the following, we outline a proof of Lemma \ref{prop:Sigma} that $\frac{\partial \Sigma}{\partial \xi_j}$, $j=1, ..., n$, is a Schwartz function. 

\begin{lemma}
\[
\begin{split}
e^{isc(\xi)}&=\cos(s|\xi|)+ic(\xi)\frac{\sin(s|\xi|)}{|\xi|}\\
\frac{\partial}{\partial \xi_j} e^{isc(\xi)}&=s\sin(s|\xi|)\frac{\xi_j}{|\xi|}+i \Big( c^j \frac{\sin(s|\xi|)}{|\xi|}+c(\xi)\frac{\partial }{\partial \xi_j}\frac{\sin(s|\xi|)}{|\xi|}\Big).
\end{split}
\]
\end{lemma}
\begin{proof}
Using the power series, we have 
\[
\begin{split}
e^{isc(\xi)}&=\sum_{n=0}^\infty \frac{\big(isc(\xi)\big)^n}{n!}\\
&=\sum \frac{\big(is c(\xi)\big)^{2n}}{(2n)!}+\sum \frac{\big( is c(\xi)\big)^{2n+1}}{(2n+1)!}\\
&=\sum \frac{(-1)^{n} s^{2n} |\xi|^{2n}}{(2n)!} +i\sum \frac{(-1)^n s^{2n+1}c(\xi)|\xi|^{2n}}{(2n+1)!}\\
&=\cos(s|\xi|)+i\frac{c(\xi)}{|\xi|}\sin(s|\xi|). 
\end{split}
\]

Let $\{e^j\}$ be an orthonormal basis of $\mathbb{R}^n$, and  denote $c^j=c(e^j)$. Then $c(\xi)=\sum_j \xi_jc^j$ for $\xi=\sum_j \xi_j e^j$. Differentiate $e^{isc(\xi)}$ with respect to $\xi_j$. 
\[
\frac{\partial}{\partial \xi_j} e^{isc(\xi)}=s\sin(s|\xi|)\frac{\xi_j}{|\xi|}+i \Big( c^j \frac{\sin(s|\xi|)}{|\xi|}+c(\xi)\frac{\partial }{\partial \xi_j}\frac{\sin(s|\xi|)}{|\xi|}\Big)
\]
\end{proof}

Notice that $\frac{\sin(y)}{y}$ is a smooth even function in term of the variable $y$. Therefore, the function 
\[
\frac{\sin(s|\xi|)}{s|\xi|}
\]
is a smooth function with respect to $s$ and $|\xi|^2$. Therefore, $\frac{\sin(s|\xi|)}{s|\xi|}$ is a smooth function with respect to $s,\xi_j$, for $j=1, \cdots, n$.

Set $H(s, \xi)=\frac{\sin(s|\xi|)}{s|\xi|}$. Then 
\[
\frac{\partial}{\partial \xi_j} e^{isc(\xi)}=s^2H(s, \xi)\xi_j+is\big(c^jH(s, \xi)+ c(\xi)\frac{\partial}{\partial \xi_j} H(s, \xi)\big). 
\]

Recall that the function $\Sigma$ is defined as follows
\[
1\otimes \chi(c(\xi)),\ \text{where}\ \chi(c(\xi))= \int_{\mathbb{R}} \widehat{\chi}(s) e^{isc(\xi)}ds.
\]
The derivative $\frac{\partial}{\partial \xi_j} \chi(c(\xi))$ can be expressed as follows,
\[
\frac{\partial}{\partial \xi_j} \chi(c(\xi))=\int_{\mathbb{R}} \widehat{\chi}(s)s \Big(sH(s, \xi)\xi_j+ i s\big(c^jH(s, \xi)+ c(\xi)\frac{\partial}{\partial \xi_j} H(s, \xi)\big)\Big)ds.
\]
We remark that as $\widehat{\chi}(s)s$ is a compactly supported smooth function, and $H(s, \xi)$ is a smooth function in both $s$ and $\xi$, the above integral formula for $\frac{\partial}{\partial \xi_j} \chi(c(\xi))$ holds true. We can even conclude that $\frac{\partial}{\partial \xi_j} \chi(c(\xi))$ is a smooth function with respect to the variable $\xi$. 

In the following, we show that $\frac{\partial}{\partial \xi_j} \chi(c(\xi))$ is a Schwartz function. 

\begin{lemma}\label{lem:h}
Let $h(y)=\frac{\sin(y)}{y}$. $h(y)$ is a smooth function on $\mathbb{R}$. Furthermore, for any $k$, there are polynomials $\phi_n(y)$ and $\psi_n(y)$ of degree less than or equal to $n$ such that 
\[
\frac{d^n}{dy^n}h(y)=\frac{\sin(y)\phi_n(y)+\cos(y)\psi_n(y)}{y^{n+1}}.
\]
\end{lemma}

\begin{proof}This can be proved by induction with direct computation. 
\end{proof}

\begin{lemma}\label{lem:H}
For every $J=(j_1,\cdots, j_m)\in \mathbb{N}\times \cdots \mathbb{N}$, there are polynomials  $\Phi_{J,k}$ and $\Psi_{J,k}$ of $n$-variables  with $2(j_1+\cdots+j_m))\geq \text{deg}(\Phi_{J,k})$ and $ (\text{deg}(\Psi_{J,k})+1)$ satisfying
\[
\frac{\partial^{j_1+\cdots +j_m}}{\partial \xi_J}H(s, \xi)=
\sum_{k=0}^{j_1+\cdots+j_m} s^k \frac{d^k}{dy^k}h(s|\xi|) \frac{\Phi_{J,k}+|\xi | \Psi_{J, k}}{|\xi|^{2(j_1+\cdots +j_m)+1}}
\]
\end{lemma}
\begin{proof}
This can be proved directly by induction on the total order of derivatives $\nu=j_1+\cdots +j_m$. 
\end{proof}

We now look at the function $\frac{\partial}{\partial \xi_j} \chi(c(\xi))$.  It is a sum of three terms 
\[
\int_{\mathbb{R}} \widehat{\chi}(s)s^2H(s, \xi)\xi_j ds, ic^j \int_{\mathbb{R}} \widehat{\chi}(s)s^2H(s, \xi)ds, i\int_{\mathbb{R}}\widehat{\chi}(s)s c(\xi)\frac{\partial}{\partial \xi_j} H(s, \xi)ds.
\]
To prove that $\frac{\partial}{\partial \xi_j} \chi(c(\xi))$ is a Schwartz function, it suffices to prove that each of them is a Schwartz function. As they are all similar, it is enough to prove that for a Schwartz function $\kappa(s)$, the following function 
\begin{equation}\label{eq:function}
|\xi|^l \int_{\mathbb{R}}\kappa(s) \frac{\partial^{j_1+\cdots +j_m}}{\partial \xi_J}H(s, \xi)ds
\end{equation}
is bounded for every fixed $j_1,\cdots, j_m$ and $l$. 

By Lemma \ref{lem:H}, we are reduced to prove for each $k$, the following function 
\begin{equation}\label{eq:integration}
|\xi|^l \int_{\mathbb{R}}\kappa(s) s^k \frac{d^k}{dy^k}h(s|\xi|) \frac{\Phi_{J,k}+|\xi | \Psi_{J, k}}{|\xi|^{2(j_1+\cdots +j_m)+1}}ds =\frac{\Phi_{J,k}+|\xi | \Psi_{J, k}}{|\xi|^{2(j_1+\cdots +j_m)+1}} |\xi|^l \int_{\mathbb{R}}\kappa(s) s^k \frac{d^k}{dy^k} h(s|\xi|)ds.
\end{equation}

Notice that $\frac{\partial^l}{\partial s^l} \frac{d^k}{dy^k} h(s|\xi|)= |\xi|^l\frac{d^k}{dy^k} h(s|\xi|)$. We have the following equation
\[
\frac{\Phi_{J,k}+|\xi | \Psi_{J, k}}{|\xi|^{2(j_1+\cdots +j_m)+1}}|\xi|^l \int_{\mathbb{R}}\kappa(s) s^k \frac{d^k}{dy^k} h(s|\xi|)ds
=\frac{\Phi_{J,k}+|\xi | \Psi_{J, k}}{|\xi|^{2(j_1+\cdots +j_m)+1}}\int_{\mathbb{R}}\kappa(s) s^k \frac{\partial^l}{\partial s^l} \frac{d^k}{dy^k} h(s|\xi|)ds.
\]
Integration by parts gives that the right hand side of the equation can be written as
\begin{equation}\label{eq:newfunction}
\frac{\Phi_{J,k}+|\xi | \Psi_{J, k}}{|\xi|^{2(j_1+\cdots +j_m)+1}} (-1)^l \int_{\mathbb{R}} \frac{d^l}{d s^l}\big(\kappa(s) s^k \big)\frac{d^k}{dy^k} h(s|\xi |)ds.
\end{equation}

By the degree counting, when $|\xi|$ is sufficiently large, 
\[
\frac{\Phi_{J,k}+|\xi | \Psi_{J, k}}{|\xi|^{2(j_1+\cdots +j_m)+1}}
\]
is uniformly bounded. 

By Lemma \ref{lem:h}, the function $\frac{d^k}{dy^k}h$ is uniformly bounded again for all $y$. Therefore, $\frac{d^k}{dy^k} h(s|\xi|)$ is uniformly bounded. Finally, as $\kappa$ is assumed to be a Schwartz functions, $\frac{d^l}{ds^l} \big(\kappa s^k\big)$ is again a Schwartz function. Therefore, the integral 
\[
\int_{\mathbb{R}} \frac{d^l}{d s^l}\big(\kappa(s) s^k \big)\frac{d^k}{dy^k} h(s|\xi |)ds
\]
is uniformly bounded. Hence, we summarize from the above discussion that the whole function
\[
\frac{\Phi_{J,k}+|\xi | \Psi_{J, k}}{|\xi|^{2(j_1+\cdots +j_m)+1}} (-1)^l \int_{\mathbb{R}} \frac{d^l}{d s^l}\big(\kappa(s) s^k \big)\frac{d^k}{dy^k} h(s|\xi |)ds.
\]
introduced in Equation (\ref{eq:newfunction}) is bounded, and therefore the function introduced in Equation (\ref{eq:function}) is bounded via Equation (\ref{eq:integration}).  From this property, we can conclude that the function $\frac{\partial}{\partial \xi_j} \chi(c(\xi))$ is a Schwartz function. And it follows that $\frac{\partial \Sigma}{\partial \xi_j}$ is a Schwartz function.

\end{document}